\pdfoutput=1
\documentclass[12pt]{article}
\usepackage[a4paper,bindingoffset=0mm,left=12mm,right=12mm,top=10mm,bottom=25mm, footskip=10mm]{geometry}
\usepackage[T1]{fontenc}
\usepackage{newtxtext,newtxmath} 
\usepackage{amsmath}
\usepackage{flexisym}

\usepackage[numbers]{natbib}
\usepackage{bbm}
\usepackage{memhfixc}

\usepackage{amsthm}
\usepackage{upgreek}
\usepackage{appendix}
\usepackage{hyperref}

\newtheorem{Definition}{Definition}[section]

\newtheorem{Lemma}{Lemma}[section]
\newtheorem{Theorem}{Theorem}[section]

\newtheorem{Proposition}{Proposition}[section]
\newtheorem{Remark}{Remark}[section]
\def \C {\mathbbm{C}}

\usepackage{graphicx}
\def\b{\ensuremath\mathbf}
\def\c{\ensuremath\mathcal}
\def\h{\ensuremath\widehat}

\DeclareRobustCommand{\rchi}{{\mathpalette\irchi\relax}}
\newcommand{\irchi}[2]{\raisebox{\depth}{$#1\chi$}}

\usepackage{float}
\usepackage{algorithm,pifont}
\usepackage{algpseudocode}

\usepackage{titlesec}
\titlespacing*{\section}
{0pt}{2ex plus 1ex minus .2ex}{4.3ex plus .2ex}
\titlespacing*{\subsection}
{0pt}{2ex plus 1ex minus .2ex}{4.3ex plus .2ex}
\titlespacing*{\subsubsection}
{0pt}{2ex plus 1ex minus .2ex}{4.3ex plus .2ex}

\title{Blind Ptychography via Blind Deconvolution}
\author{Mark Philip Roach}%\email{roachma3@msu.edu}

\begin{document}

\setlength{\belowcaptionskip}{10pt}
\setlength{\belowdisplayskip}{20pt} %\setlength{\belowdisplayshortskip}{10pt}
\setlength{\abovedisplayskip}{10pt} %\setlength{\abovedisplayshortskip}{0pt}


\maketitle
%\affil{\orgdiv{Department of Mathematics}, \orgname{Michigan State University}, \orgaddress{\city{East Lansing}, \postcode{48824}, \state{Michigan}, \country{USA}}}


\abstract{Ptychography involves a sample being illuminated by a coherent, localised probe of illumination. When the probe interacts with the sample, the light is diffracted and a diffraction pattern is detected. Then the probe or sample is shifted laterally in space to illuminate a new area of the sample while ensuring there is sufficient overlap. \textbf{Far-field Ptychography} occurs when there is a large enough distance (when the Fresnel number is $\ll 1$) to obtain magnitude-square Fourier transform measurements. In an attempt to remove ambiguities, masks are utilized to ensure unique outputs to any recovery algorithm are unique up to a global phase. In this paper, we assume that both the sample and the mask are unknown, and we apply blind deconvolutional techniques to solve for both. Numerical experiments demonstrate that the technique works well in practice, and is robust under noise.

%\keywords{Blind Ptychography, Blind Deconvolution}

\section{Introduction}
%\textcolor{red}{Prove Ambiguities}\\
Ptychography involves a sample being illuminated by a coherent, localised probe of illumination. When the probe interacts with the sample, the light is diffracted and a diffraction pattern is detected. Then the probe or sample is shifted laterally in space to illuminate a new area of the sample while ensuring there is sufficient overlap.
\textbf{Far-field Ptychography} occurs when there is a large enough distance (when the Fresnel number is $\ll 1$) to obtain magnitude-square Fourier transform measurements.\\
\indent Ptychography was initially studied in the late 1960s (\cite{hoppe1969diffraction}), with the problem solidified in 1970 (\cite{hegerl1970dynamische}). The name "Ptychography" was coined in 1972 (\cite{hegerl1972phase}), after the Greek word \textit{to fold} because the process involves an interference pattern such that the scattered waves fold into one anotherthe (coherent) Fourier diffraction pattern of the object. Initially developed to study crystalline objects under a scanning transmission electron microscope, since then the field has widen to setups such as using visible light (\cite{clark2014continuous}, \cite{huang2015fly}, \cite{odstrvcil2018arbitrary}), x-rays (\cite{edo2013sampling}, \cite{tsai2016x}, \cite{pfeiffer2018x}), or electrons (\cite{yang2016simultaneous},\cite{gao2017electron}\cite{jiang2018electron}). It is benefited from being  unaffected by lens-induced aberrations or diffraction effects unlike conventional lens imaging. Various types of ptychography are studied based on the optical configuration of the experiments. For instance, Bragg Ptychography (\cite{godard2011imaging} \cite{takahashi2013bragg}, \cite{hruszkewycz2017high}, \cite{li2021revealing}) measures strain in crystalline specimens by shifting the surface of the specimen.\\
\indent Fourier ptychography (\cite{zheng2013wide},\cite{tian2014multiplexed},\cite{ou2015high},\cite{zheng2021concept}) consists of taking multiple images at a wide field-of-view then computationally synthesizing into a high-resolution image reconstruction in the Fourier domain. This results in an increased resolution compared to a conventional microscope. \begin{figure}[H]
\centering
\includegraphics[width=1\textwidth]{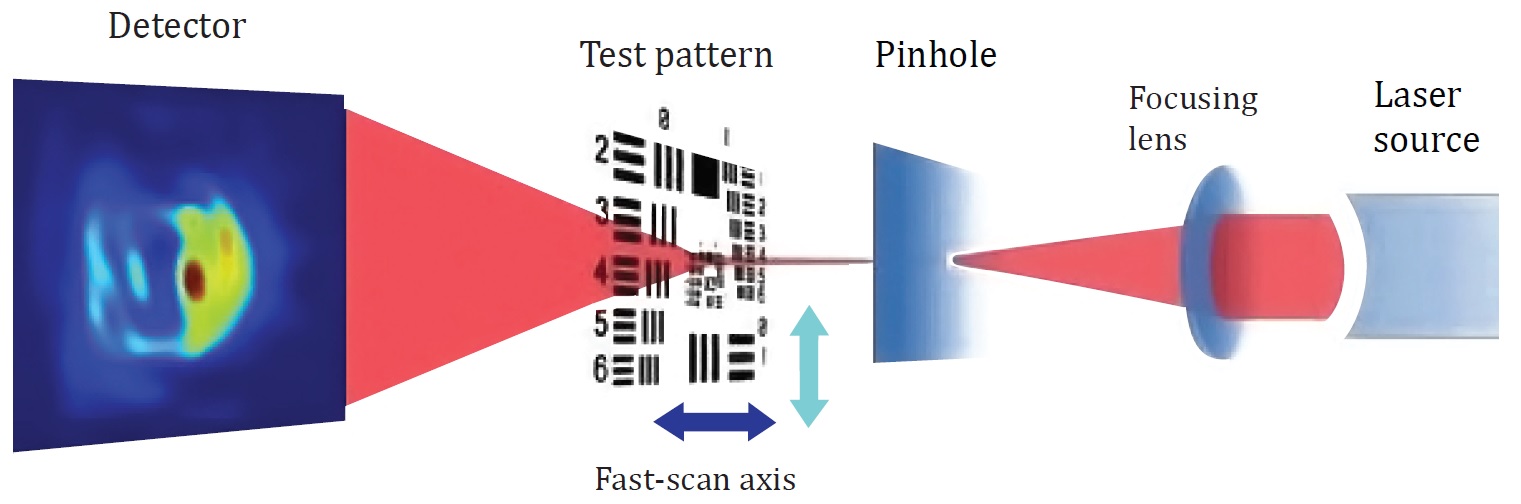}
\caption{\cite{huang2015fly} Experimental setup for fly-scan ptychography}
\end{figure}
\section{Far-field Fourier Ptychography} \label{sec: fffp}
Let $\mathbf{x}, \mathbf{m} \in \mathbb{C}^{d}$ denote the unknown sample and known mask, respectively. We suppose that we have $d^2$ noisy ptychographic measurements of the form
\begin{align} \label{eqn: far field measurements}
(\mathbf{Y})_{\ell,k} = |(\mathbf{F}(\mathbf{x} \circ S_k \mathbf{m}))_\ell|^2 + \; (\mathbf{N})_{\ell,k}, \quad (\ell,k) \in [d]_0 \times [d]_0,
\end{align}
where $S_k, \circ, \mathbf{F} := \mathbf{F}_d$ denote $k^{\text{th}}$ circular shift, Hadamard product, and $d$-dimensional discrete Fourier transform, and $\mathbf{N}$ is the matrix of additive noise.\\
\indent In this section, we will define a discrete Wigner distribution deconvolution method for recovering a discrete signal.  A modified Wigner distribution deconvolution approach is used to solve for an estimate of  $\hat{\b{x}}\hat{\b{x}}^* \in \mathbb{C}^{d \times d}$ and then angular synchronization is performed to compute estimate of $\hat{\b{x}}$ and thus $\b{x}$.\\
\indent In Section \ref{sec: Properties of the Discrete Fourier Transform}, we introduce definitions and techincal lemmas which will be of use. In particular, the decoupling lemma (Lemma \ref{lem: Decoupling}) allows use to effectively 'separate' the mask and object from a convolution. In Section \ref{sec: Discretized Wigner Distribution Deconvolution}, these technical lemmas are applied to the ptychographic measurements to write the problem as a decoupled deconvolution problem, the blind variant of which will be studied later on. In Section \ref{sec:Wigner Distribution Deconvolution Algorithm}, an additional Fourier transform is applied and the measurements have been rewritten to a form in which a pointwise division approach can be applied. Sub-sampled version of this theorem are also given. We then state the full algorithm for recovering the sample.
\subsection{Properties of the Discrete Fourier Transform} \label{sec: Properties of the Discrete Fourier Transform}
We firstly define the modulation operator.
%First we introduce some basic notations and definitions involving the discrete Fourier transform.
%\begin{Definition} Let $\mathbf{x} = (x_0, x_1, \ldots, x_{d-1})^T \in \mathbb{C}^{d}$, $[d]_0 = \{0,1,\ldots, d-1\}$ the set of remainders modulo $d$.\\
%(i) We define the \textbf{support} of $\mathbf{x}$, $supp(\mathbf{x})$,  as the index set corresponding to the non-zero entries of $\mathbf{x}$, that is,
%\begin{align}
%supp(\mathbf{x}) := \{n \in [d]_0 \mid x_n \neq 0\}
%\end{align}
%(ii) In the previous section, we saw the Fourier transform. The \textbf{inverse Fourier transform} is computed via
%\begin{align}
%x_n = (F_{d}^{-1} \hat{\mathbf{x}})_n = \dfrac{1}{d} \sum_{k=0}^{d-1} \hat{x}_k e^{2\pi i k n/d}
%\end{align}
%(iii) Given a vector $\mathbf{x} \in \mathbb{C}^{d}$, denote by $\tilde{\mathbf{x}}$ the \textbf{reversal} of $\mathbf{x}$ about its first entry so that its components are 
%\begin{align}
%\tilde{x}_n := x_{-n \; mod \; d}, \quad \forall n %\in [d]_0
%\end{align}
%(iv) Given $k \in [d]_0$, define the %\textbf{modulation operator} $W_k : \mathbb{C}^{d} \longrightarrow \mathbb{C}^{d}$ component-wise via
%\begin{align}
%(W_k \mathbf{x})_n = x_n e^{2\pi i k n/d}, \quad \forall n \in [d]_0
%\end{align}
%\end{Definition}
\begin{Definition}
Given $k \in [d]_0$, define the \textbf{modulation operator} $W_k : \mathbb{C}^{d} \longrightarrow \mathbb{C}^{d}$ component-wise via
\begin{align}
(W_k \mathbf{x})_n = x_n e^{2\pi i k n/d}, \quad \forall n \in [d]_0.
\end{align}
\end{Definition}
From this definition, we can develop some useful equalities which we will use in the main proofs of this section.
\begin{Lemma} \textbf{(Technical Equalities)}  (Lemma 1.3.1., \cite{merhi2019phase}) The following equalities hold for all $\mathbf{x} \in \mathbb{C}^{d}, [\ell] \in [d]_0$:\\
(i) $\b{F}_d \hat{\mathbf{x}} = d \cdot \tilde{\mathbf{x}}$;\\
(ii) $\b{F}_d (W_\ell \mathbf{x}) = S_{-\ell} \hat{\mathbf{x}}$;\\
(iii) $\b{F}_d  (S_\ell \b{x}) = W_\ell \h{\b{x}}$;\\
(iv) $W_{-\ell} \b{F}_d (S_\ell \bar{\tilde{\b{x}}}) = \bar{\hat{\b{x}}}$;\\
(v) $\overline{\widetilde{S_\ell \b{x}}} = S_{-\ell} \bar{\tilde{\b{x}}}$;\\
(vi) $\b{F}_d \bar{\b{x}} = \overline{\b{F}_d \tilde{\b{x}}}$;\\
(vii) $\tilde{\hat{\b{x}}} = \b{F}_d \tilde{\b{x}}$.
\end{Lemma}
%We now need a couple more definitions, which are operations performed on two vectors.
%\begin{Definition} Let $\b{x}, \b{y} \in \mathbb{C}^{d}, \ell \in [d]_0$.\\
%(i) We define the \textbf{circular convolution} operator, $*_{d}$ or just simply $*$, by
%\begin{align}
%(\b{x} *_d \b{y})_\ell := \sum_{n=0}^{d-1} x_n y_{(\ell - n)} \; mod \; d \in \mathbb{C}^{d}
%\end{align}
%(ii) We define the \textbf{Hadamard product}, $\circ$, by
%\begin{align}
%(\b{x} \circ \b{y})_\ell := x_\ell y_\ell \in %\mathbb{C}
%\end{align}
%\end{Definition}
We wish to be able to convert between the convolution and the Hadamard product, so we will need this useful theorem.
\begin{Theorem} \textbf{(Discretized Convolution Theorem)} (Lemma 1.3.2., \cite{merhi2019phase}) Let $\b{x}, \b{y} \in \mathbb{C}^{d}$. We have that\\
(i) $F_{d}^{-1} (\hat{\b{x}} \circ \hat{\b{y}}) = \b{x} *_d \b{y}$;\\
(ii) $(\b{F}_d \b{x}) *_d (\b{F}_d \b{y}) = d \cdot \b{F}_d (\b{x} \circ \b{y})$.
\end{Theorem}
Currently, the measurements we are dealing with will be having the specimen and the mask intertwined. We introduce the decoupling lemma to essentially detangle the two.
\begin{Lemma} \textbf{(Decoupling Lemma)} (Lemma 1.3.3., \cite{merhi2019phase}) \label{lem: Decoupling}

Let $\b{x}, \b{y} \in \mathbb{C}^{d}, \ell, k \in [d]_0$. Then
\begin{align}
\bigg( (\b{x} \circ S_{-\ell} \b{y}) *_d (\bar{\tilde{\b{x}}} \circ S_\ell \bar{\tilde{\b{y}}})\bigg)_k = \bigg( (\b{x} \circ S_{-k} \bar{\b{x}}) *_d (\tilde{y} \circ S_k \bar{\tilde{\b{y}}})\bigg)_\ell.
\end{align}
\end{Lemma}
\begin{proof} Let $\b{x}, \b{y} \in \mathbb{C}^{d}, \ell, k \in [d]_0$. By the definitions of the circular convolution, Hadamard product and shift operator, we have that
\begin{align}\nonumber
\bigg( (\b{x} \circ S_{-\ell} \b{y}) *_d (\bar{\tilde{\b{x}}} \circ S_\ell \bar{\tilde{\b{y}}})\bigg)_k &= \sum_{n=0}^{d-1}(\b{x} \circ S_{-\ell} \b{y}) _n((\bar{\tilde{\b{x}}} \circ S_\ell \bar{\tilde{\b{y}}})_{k-n}&\\\nonumber
&= \sum_{n=0}^{d-1} x_n y_{n-\ell} \bar{\tilde{x}}_{k-n} \bar{\tilde{y}}_{\ell + k -n}&\\
&= \sum_{n=0}^{d-1} x_n \bar{x}_{n-k} \tilde{y}_{\ell - n} \bar{\tilde{y}}_{k + \ell - n}&\\\nonumber
&= \sum_{n=0}^{d-1}(\b{x} \circ S_{-k} \b{x}) _n((\tilde{\b{y}} \circ S_k \bar{\tilde{\b{y}}})_{\ell-n}&\\\nonumber
&= \bigg( (\b{x} \circ S_{-k} \bar{\b{x}}) *_d (\tilde{y} \circ S_k \bar{\tilde{\b{y}}})\bigg)_\ell.&
\end{align}
\end{proof}
Lastly before entering the main part of this subsection, we need a lemma involving looking at how the Fourier squared magnitude measurements will relate to a convolution. 
\begin{Lemma} Let $\b{x} \in \mathbb{C}^{d}$. We have that
\begin{align}
|\b{F}_d \b{x}|^2 = \b{F}_d (\b{x} *_d \bar{\tilde{\b{x}}}).
\end{align}
\end{Lemma}
\begin{proof} Let $\b{x} \in \mathbb{C}^{d}$. Then we have that
\begin{align}
|\b{F}_d \b{x}|^2 = (\b{F}_d \b{x}) \circ \overline{(\b{F}_d \b{x})} = (\b{F}_d \b{x}) \circ (\b{F}_d \bar{\tilde{\b{x}}}) = \b{F}_d (\b{x} *_d \bar{\tilde{\b{x}}}).
\end{align}
\end{proof}
\subsection{Discretized Wigner Distribution Deconvolution} \label{sec: Discretized Wigner Distribution Deconvolution}
We now prove the Discretized Wigner Distribution Deconvolution theorem which will allow us to convert the measurements into a form in which we can algorithmically solve.
\begin{Theorem} (Lemma 1.3.5., \cite{merhi2019phase}) \label{thm: BlindPtych Measurements} Let $\mathbf{x}, \b{m} \in \mathbb{C}^{d}$ denote the unknown specimen and known mask, respectively. Suppose we have $d^2$ noisy ptychographic measurements of the form
\begin{align}
(\b{y}_\ell)_k = \bigg| \sum_{n=0}^{d-1} x_n m_{n - \ell} e^{-2\pi i n k/d}\bigg|^2 + (\mathbf{N})_{\ell,k}, \quad (\ell,k) \in [d]_0 \times [d]_0.
\end{align}
Let $\b{Y} \in \mathbb{R}^{d \times d}, \b{N} \in \mathbb{C}^{d \times d}$ be the matrices whose $\ell^{th}$ column is $\b{y}_\ell, \mathbf{N}_{\ell}$ respectively. Then for any $k \in [d]_0$,
\begin{align}
\bigg( \b{Y}^T \b{F}_{d}^{T}\bigg)_k = d \cdot (\b{x} \circ S_k \bar{\b{x}}) *_d  (\tilde{\b{m}} \circ S_{-k} \bar{\tilde{\b{m}}})  + \bigg( \b{N}^T \b{F}_{d}^{T}\bigg)_k.
\end{align}
\end{Theorem}
\begin{proof} Let $\ell \in [d]_0$. We have that
\begin{align}
\b{y}_\ell = |F_d (\b{x} \circ S_{-\ell} \b{m})|^2 + \mathbf{N}_{\ell} = F_d \bigg( (\b{x} \circ S_{-\ell}\b{m}) *_d (\bar{\tilde{\b{x}}} \circ S_\ell \bar{\tilde{\b{m}}})\bigg) + \mathbf{N}_{\ell}.
\end{align}
Taking Fourier transform of both sides at $k \in [d]_0$ and using that $\b{F}_d \hat{\mathbf{x}} = d \cdot \tilde{\mathbf{x}}$ yields
\begin{align}
(\b{F}_d \b{y}_\ell)_k &= d \cdot \bigg( (\b{x} \circ S_{-\ell}\b{m}) *_d (\bar{\tilde{\b{x}}} \circ S_\ell \bar{\tilde{\b{m}}})\bigg)_{-k} + (\b{F}_d \mathbf{N}_{\ell})_k &\\\nonumber
&= d \cdot \bigg( (\b{x} \circ S_{k}\overline{\b{x}}) *_d (\tilde{\b{m}} \circ S_{-k} \bar{\tilde{\b{m}}})\bigg)_{\ell} + (\b{F}_d\mathbf{N}_{\ell})_k&,
\end{align}
by previous lemma. For fixed $\ell \in [d]_0$, the vector $\b{F}_d \b{y}_\ell \in \mathbb{C}^{d}$ is the $\ell^{\text{th}}$ column of the matrix $\b{F}_d \b{Y}$, thus its transpose $\b{y}_{\ell}^T \b{F}_{d}^T \in \mathbb{C}^{d}$ is the $\ell^{\text{th}}$ row of the matrix $(\b{F}_d \b{Y})^T$. Similarly, $((\mathbf{N})_{\ell})^{T} \b{F}_{d}^{T} \in \mathbb{C}^{d}$ is the $\ell^{\text{th}}$ row of $(\b{F}_d \b{N})^T$. Thus we have that
\begin{align}
\bigg(\bigg(\b{Y}^T \b{F}_{d}^{T}\bigg)_k\bigg)_\ell = d \cdot \bigg( (\b{x} \circ S_k \bar{\b{x}}) *_d  (\tilde{\b{m}} \circ S_{-k} \bar{\tilde{\b{m}}})\bigg)_\ell + \bigg(\b{N}^T \b{F}_{d}^{T}\bigg)_{k,\ell}.
\end{align}
Thus we have that
\begin{align}
\bigg( \b{Y}^T \b{F}_{d}^{T}\bigg)_k = d \cdot (\b{x} \circ S_k \bar{\b{x}}) *_d  (\tilde{\b{m}} \circ S_{-k} \bar{\tilde{\b{m}}}) + \bigg( \b{N}^T \b{F}_{d}^{T}\bigg)_k.
\end{align}
\end{proof}
We note that $\b{x} \circ S_k \bar{\b{x}}$ is a diagonal of $\b{x}\b{x}^*$.
\subsection{Wigner Distribution Deconvolution Algorithm} \label{sec:Wigner Distribution Deconvolution Algorithm} 
We suppose that the mask is known and the specimen is unknown. By taking an additional Fourier transform and using the discretized convolution theorem, we have these variances of the previous lemmas.
\begin{Theorem} \textbf{(Discretized Wigner Distribution Deconvolution)} Let $\mathbf{x}, \b{m} \in \mathbb{C}^{d}$ denote the unknown specimen and known mask, respectively. Suppose we have $d^2$ noisy spectrogram measurements of the form
\begin{align}
(\b{y}_\ell)_k = \bigg| \sum_{n=0}^{d-1} x_n m_{n - \ell} e^{-2\pi i n k/d}\bigg|^2 + (\mathbf{N})_{\ell,k} , \quad (\ell,k) \in [d]_0 \times [d]_0.
\end{align}
Let $\b{Y} \in \mathbb{R}^{d \times d}$ be the matrix whose $\ell^{th}$ column is $\b{y}_\ell$. Then for any $k \in [d]_0$
\begin{align}
\b{F}_d \bigg( \b{Y}^T \b{F}_{d}^T\bigg)_k = d \cdot \b{F}_d (\b{x} \circ S_k \bar{\b{x}}) \circ \b{F}_d (\tilde{\b{m}} \circ S_{-k} \bar{\tilde{\b{m}}}) + \b{F}_d \bigg( \b{N}^T \b{F}_{d}^T\bigg)_k.
\end{align}
\end{Theorem}
We also have a similar result based on the work in Appendix \ref{BP Appendix}.
\begin{Lemma}\textbf{(Sub-Sampling In Frequency)} Suppose that the spectrogram measurements are collected on a subset $\mathcal{K} \subseteq [d]_0$ of $K$ equally spaced Fourier modes. Then for any $\omega \in [K]_0$
\begin{align*}
\b{F}_d \bigg((\b{Y}_{K,d})^T \b{F}_{K}^T\bigg)_\omega = K\sum_{r = 0}^{\frac{d}{K} - 1} \b{F}_d (\b{x} \circ S_{\ell L - \alpha}\bar{\b{x}}) \circ \b{F}_d (\tilde{\b{m}} \circ S_{\alpha - \ell L} \bar{\tilde{\b{m}}}) + \b{F}_d \bigg((\b{N}_{K,d})^T \b{F}_{K}^T\bigg)_\omega
\end{align*}
where $\b{Y}_{K,d} \in \mathbb{C}^{K \times d} $ is the matrix of sub-sampled noiseless $K \cdot d$ measurements.
\end{Lemma}
\begin{Lemma} \textbf{(Sub-Sampling In Frequency And Space)} Suppose we have spectrogram measurements collected on a subset $\mathcal{K} \subseteq [d]_0$ of $K$ equally spaced frequencies and a subset $\mathcal{L} \subseteq [d]_0$ of $L$ equally spaced physical shifts. Then for any $\omega \in [K]_0, \alpha \in [L]_0$
\begin{align*}
\bigg(\b{F}_L (\b{Y}_{K,L})^T (\b{F}_{K}^T)_\omega\bigg)_\alpha &= \dfrac{KL}{d^3} \sum_{r = 0}^{\frac{d}{K} - 1} \sum_{\ell = 0}^{\frac{d}{L} - 1} \bigg( \b{F}_d (\hat{\b{x}} \circ S_{\ell L - \alpha}\bar{\hat{\b{x}}})\bigg)_{\omega - rK} \bigg(F_d (\hat{\b{m}} \circ S_{\alpha - \ell L} \bar{\hat{\b{m}}})\bigg)_{\omega - rK}&\\
&+ \bigg(\b{F}_L (\b{N}_{K,L})^T (\b{F}_{K}^T)_\omega\bigg)_\alpha&,
\end{align*}
where $\b{Y}_{K,L} \in \mathbb{C}^{K \times L}$ is the matrix of sub-sampled noiseless $K \cdot L$ measurements.
\end{Lemma}
Assume that $\b{m}$ is band-limited with $supp(\hat{\b{m}}) = [\delta]_0$ for some $\delta \ll d$. Then the algorithm below allows for the recovery of an estimate of $\hat{\b{x}}$ from spectrogram measurements via Wigner distribution deconvolution and angular synchronization.

\begin{algorithm}[H]
\caption{(Algorithm 1, \cite{merhi2019phase}) Wigner Distribution Deconvolution Algorithm} \label{alg: WDD Ptych}
\begin{algorithmic}
\Require 1) $Y_{d,L} \in \mathbb{C}^{d \times L}$, matrix of noisy measurements.\\
2) Mask $\b{m} \in \mathbb{C}^{d}$ with $supp(\hat{\b{m}}) = [\delta]$.\\
3) Integer $\kappa \leq \delta$, so that $2\kappa - 1$ diagonals of $\hat{\b{x}}\hat{\b{x}}^*$ are estimated, and $L = \delta + \kappa - 1$.\\
\Ensure An estimate $\mathbf{x}_{est}$ of $\b{x}$ up to a global phase.
\State 1) Perform pointwise division to compute
\begin{align}
\dfrac{1}{d}\dfrac{\b{F}_d \bigg( \b{Y}^T \b{F}_{d}^T\bigg)_k}{\b{F}_d^{-1} (\tilde{\b{m}} \circ S_{-k} \bar{\tilde{\b{m}}})}.
\end{align}
\State 2) Invert the $(2\kappa - 1)$ Fourier transforms above.
\State 3) Organize values from step 2 to form the diagonals of a banded matrix $Y_{2\kappa - 1}$.
\State 4) Perform angular synchronization on $Y_{2\kappa - 1}$ to obtain $\hat{\mathbf{x}}_{est}$.
\State 5) Let $\mathbf{x}_{est} = \b{F}_{d}^{-1}\hat{\mathbf{x}}_{est}$.
\end{algorithmic}
\end{algorithm}
When the mask is known with $supp(\hat{\b{m}}) = [\delta]_0$, $\delta \ll d$, maximum error guarantees (Theorem 2.1.1., \cite{merhi2019phase}) are given depending on $\b{x}$, $d$, $\kappa$, $L$, $\Vert N_{d,L}\Vert_{F}$ (the matrix formed by the noise) and the mask dependent constant $\mu > 0$,
\begin{align}
\mu := \underset{|p| \leq \kappa - 1, q \in [d]_0}{\min} \bigg| (F_d (\hat{\b{m}} \circ S_p \bar{\hat{\b{m}}}))_{q}\bigg|.
\end{align}
In the next section, we look at the situation in which both the specimen and mask are unknown. Since we have already shown that we can rewrite the Fourier squared magnitude measurements as convolutions between the shifted autocorrelations, then the next obvious step is when both the specimen and mask are unknown. This is the topic of \textbf{blind deconvolution}, which seeks to recover vectors from their deconvolution. In particular, we will look at a couple of approaches which involve making assumptions based on real world applications.
\section{Blind Deconvolution} \label{sec: Blind Deconvolution}
\subsection{Introduction}
%\textcolor{red}{Prove ambiguities}

Blind Deconvolution is a problem that has been mathematically considered for decades, from more past work (\cite{ayers1988iterative}, \cite{katsaggelos1991maximum}, \cite{thiebaut1995strict},
\cite{fish1995blind}\cite{molina1997bayesian}, \cite{kundur1998novel})
to more recent work (\cite{carasso2001direct}, \cite{likas2004variational}, \cite{ahmed2013blind}, \cite{fortunato2014fast} \cite{li2019rapid}), and summarized in \cite{levin2011understanding}. The goal is to recover a sharp image from an initial blurry image. The first application to compressive sensing was considered in
\cite{ahmed2013blind}. \\
\indent We consider one-dimensional, discrete, noisy measurements of the form $\b{y} = \b{f} * \b{g} + \b{n}$, where the $\b{f}$ is considered to be an object, signal, or image of consideration. $\b{g}$ is considered to be a blurring, masking, or point-spread function. $\b{n}$ is considered to be the noise vector. $*$ refers to circulant convolution\footnote{$*$ should refer to ordinary convolution, but for $\b{g}$ that will be considered in this chapter, circulant convolution will be sufficient.}. We consider situations when both $\mathbf{f}$ and $\mathbf{g}$ are unknown. The process of recovering the object and blurring function can be generalized to two-dimensional measurements.\\
\indent The problem of estimating the unknown blurring function and unknown object simultaneously is known as \textbf{blind image restoration} (\cite{you1999blind}, \cite{wu2022blind}, \cite{kundur1996blind}, \cite{sroubek2003multichannel}). Although strictly speaking, \textbf{blind deconvolution} refers to the noiseless model of recovering $\mathbf{f}$ and $\mathbf{g}$ from $\mathbf{y} = \mathbf{f} * \mathbf{g}$, the noisy model is commonly refered to as blind deconvolution itself, and this notation will be continued in this chapter.\\
\indent As we will show later, the problem is ill-posed and ambiguities lead to no unique solution to the pair being viable from any approach.\\
\indent In Section \ref{sec: Blind Deconvolution Model}, we consider the underlying measuerements and assumptions that we will consider. We then show how through manipulation, that we can re-write the original problem as a minimization of a non-convex linear function. In Section \ref{sec: Wirtinger Gradient Descent}, we demonstrate an iterative approach to the minimization problem, in particular, by applying Wirtinger Gradient Descent. In Section \ref{sec: Blind Deconvolution Algorithms}, we outline the initial estimate used for this gradient descent and fully layout the algorithm that will apply in our numerical simulations. In Section \ref{sec: Blind Deconvolution Main Theorems}, we will look at the recovery guarantees that currently exist for this approach. Finally, in Section \ref{sec: Blind Deconvolution Key Conditions}, we consider the key conditions used to generate the main recovery theorem, and where further work could be done to generalize these conditions and ultimately, allow more guarantees of recovery.
\subsection{Blind Deconvolution Model} \label{sec: Blind Deconvolution Model}
We now want to approach the blind ptychography problem in which both the mask and specimen are unknown. Using the lemmas in the previous section, we can see that this would reduce to solving a blind deconvolution problem.
\begin{Definition} We consider the blind deconvolution model
\begin{align*}
\b{y}' = \b{f} * \b{g} + \b{n}, \quad \b{y}, \b{f}, \b{g}, \b{n} \in \mathbb{C}^d,
\end{align*}
where $\b{y}$ are blind deconvolutional measurements, $\b{f}$ is the unknown blurring function (which serves a similar role as to our phase retrieval masks), $\b{n}$ is the noise, and $\b{g}$ is the signal (which serves a similar role as to our phase retrieval object). Here $*$ denotes circular convolution. 
\end{Definition}
 We will base our work on the algorithm suggested in \cite{li2019rapid}, considering the assumptions used. In \cite{li2019rapid}, the authors impose general conditions on $\b{f}$ and $\b{g}$ that are not restricted to any particular application but allows for flexibility.  They also assume that $\b{f}$ and $\b{g}$ belong to known linear subspaces.\\
\indent For the blurring function, it is assumed that $\b{f}$ is either compactly supported, or that $\b{f}$ decays sufficiently fast so that it can be well approximated by a compactly supported function. Therefore, we make the assumption that $\b{f} \in \mathbb{C}^d$ satisfies
\begin{align*}
\b{f} := 
\begin{bmatrix}
\b{h}\\
\b{0}_{d-K}\\
\end{bmatrix},
\end{align*} 
for some $k \ll d, \b{h} \in \mathbb{C}^K$. This again reinforces the notion that the blurring function is analogous to our masking function since both are compactly supported. 
\begin{figure}[H]
\centering
\includegraphics[width=1\textwidth]{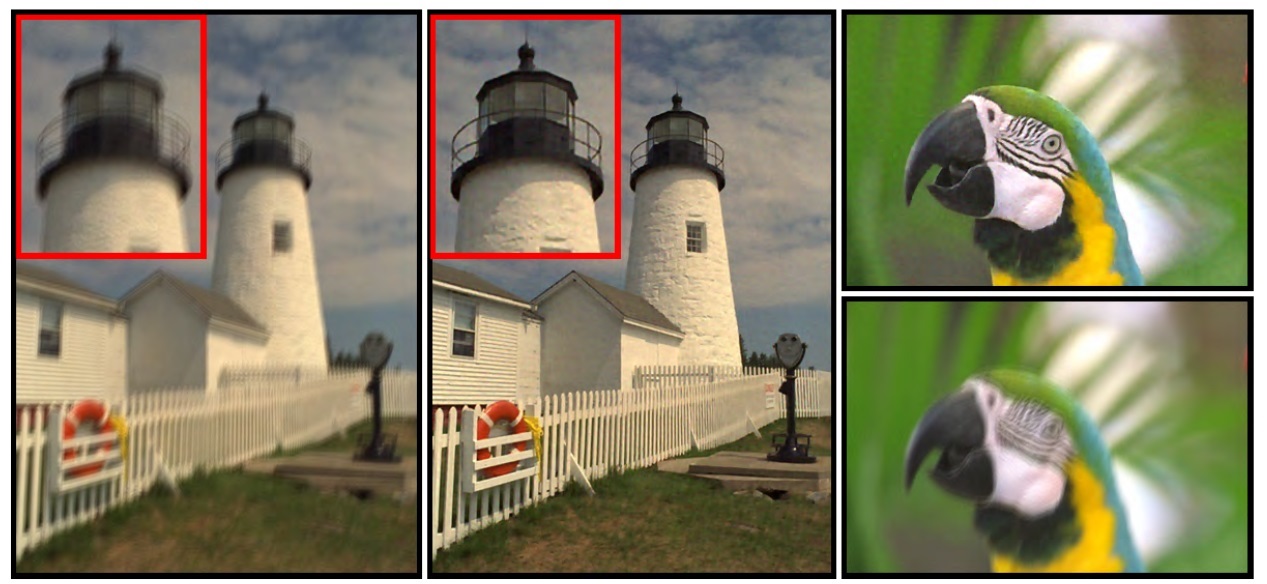}
\caption{\cite{fortunato2014fast} An example of an image deblurring by solving the deconvolution} 
\end{figure}
For the signal, it is assumed that $\b{g}$ belongs to a linear subspace spanned by the columns of a known matrix $\b{C}$, i.e., $\b{g} = \b{C}\bar{\b{x}}$ for some matrix $\b{C} \in \mathbb{C}^{d \times N}, N \ll d$. This will lead to an additional restriction we have to place on our blind ptychography but one for which there are real world applications for which this assumption makes reasonable sense.\\
\indent In \cite{li2019rapid}, the authors use that $\b{C}$ is a Gaussian random matrix for theoretical guarantees although they demonstrated in numerical simulations that this assumption is not necessary to gain results. In particular, they found good results for when $\b{C}$ represents a wavelet subspace (suitable for images) or when $\b{C}$ is a Hadamard-type matrix (suitable for communications).\\
\indent We assume the noise is complex Gaussian, i.e. $\b{n} \sim \mathcal{N}(0, \tfrac{\sigma^2 L_{0}^2}{2}I_d) + i\mathcal{N}(0, \tfrac{\sigma^2 L_{0}^2}{2}I_d)$ is a complex Gaussian noise vector, where $L_0 = \Vert\b{h}_0\Vert \cdot \Vert\b{x}_0\Vert$, and $\b{h}_0, \b{x}_0$ are the true blurring function and signal. $\sigma^{-2}$ represents the SNR.\\
\indent The goal is convert the problem into one which can be algorithmically solvable via gradient descent.
\begin{Proposition} \cite{li2019rapid} \label{prop: newblindmodel}
Let $\b{F}_{d} \in C^{d \times d}$ be DFT matrix. Let $\b{B} \in \mathbb{C}^{d \times K}$ denote the first $K$ columns of $\b{F}_{d}$. Then we have that
\begin{align}
\b{y} = \b{Bh} \circ \overline{\b{Ax}} + \b{e},
\end{align}
where $\b{y} = \dfrac{1}{\sqrt{d}}\widehat{\b{y}'}$, $\bar{\b{A}} = \b{FC} \in \mathbb{C}^{d \times N}$, and $\b{e} = \dfrac{1}{\sqrt{d}}\b{F_dn}$ represents noise.
\end{Proposition}
\begin{proof}
By the unitary property of $\b{F}_d$, we have that $\b{B}^*\b{B} = \b{I}_K$.By applying the scaled DFT matrix $\sqrt{L}\b{F}_{d}$ to both sides of the convolution, we have that 
\begin{align*}
\sqrt{L}\b{F_dy} = (\sqrt{L}\b{F_df}) \circ (\sqrt{L}\b{F_dg}) + \sqrt{L}\b{F_dn}.
\end{align*}
Additionally, we have that
\begin{align*}
\b{F_df} =
\begin{bmatrix}
\b{B} & \b{M}\\
\end{bmatrix}
\begin{bmatrix}
\b{h}\\
\b{0}_{d-K}\\
\end{bmatrix}
= \b{Bh},
\end{align*}
and we let $\bar{\b{A}} = \b{FC} \in \mathbb{C}^{L \times N}$. Since $\b{C}$ is Gaussian, then $\bar{\b{A}} = \b{FC}$ is also Gaussian. In particular,
\begin{align*}
\bar{\b{A}}_{ij} \sim \mathcal{N}(0, \tfrac{1}{2}) + i\mathcal{N}(0, \tfrac{1}{2}).
\end{align*}
Thus by dividing by $d$, the problem converts to 
\begin{align*}
\dfrac{1}{\sqrt{d}}\widehat{\b{y}'} = \b{Bh} \circ\overline{\b{Ax}} + \b{e},
\end{align*}
where $\b{e} = \dfrac{1}{\sqrt{d}}\b{F_Ln} \sim \mathcal{N}(0, \tfrac{\sigma^2 L_{0}^2}{2L}\b{I}_d) + i\mathcal{N}(0, \tfrac{\sigma^2 L_{0}^2}{2L}\b{I}_d)$ serves as complex Gaussian noise. Hence by letting $\b{y} = \dfrac{1}{\sqrt{d}}\widehat{\b{y}'}$, we arrive at
\begin{align*}
\b{y} = \b{Bh} \circ \overline{\b{Ax}} + \b{e}.
\end{align*}
\end{proof}
We have thus transformed the original blind deconvolution model as a hadamard product. This form of the problem is used in the rest of the section, where $\b{y} \in \mathbb{C}^{d}, \b{B} \in \mathbb{C}^{d \times K}, \b{A} \in \mathbb{C}^{d \times N}$ are given. Our goal is to recover $\b{h}_0$ and $\b{x}_0$. \\
\indent There are inherent ambiguities to the problem however. If $(\b{h}_0, \b{x}_0)$ is a solution to the blind deconvolution problem, then so is $(\alpha \b{h}_0, \alpha^{-1} \b{x}_0)$ for any non-zero constant $\alpha$. For most real world applications, this is not an issue. Thus for uniformity, it is assumed that $\Vert\b{h}_0\Vert = \Vert\b{x}_0\Vert = \sqrt{L_0}$.
\begin{Definition}
We define the matrix-valued linear operator $\mathcal{A} : \mathbb{C}^{K \times N} \longrightarrow \mathbb{C}^d$ by
\begin{align*}
\c{A}(Z) := \{\b{b}_{\ell}^{*}Z \b{a}_\ell\}_{\ell=1}^{d},
\end{align*}
where $\b{b}_k$ denotes the $k$-th column of $\b{B}^*$, and $\b{a}_k$ is the $k$-th column of $\b{A}^*$. We also define the corresponding adjoint operator $\c{A}^* : \mathbb{C}^d \longrightarrow \mathbb{C}^{K \times N}$, given by
\begin{align*}
\mathcal{A}^* (\b{z}) := \sum_{k=1}^{d}\b{z}_k  \b{b}_k \b{a}_{k}^*.
\end{align*}
\end{Definition}
We see that this translates to a lifting problem, where
\begin{align*}
\sum_{k=1}^{d} \b{b}_{\ell}\b{b}_{\ell}^* = \b{B}^* \b{B} = \b{I}_K, \quad \Vert\b{b}_\ell\Vert = \dfrac{K}{d}, \quad \mathbb{E}(\b{a}_{\ell}\b{a}_{\ell}^*) = \b{I}_N, \quad \forall k \in [d].
\end{align*}
\begin{Lemma} \label{lem: newblindmodel}
Let $y$ be defined as in Proposition \ref{prop: newblindmodel}. Then
\begin{align}
\b{y} = \mathcal{A}(\b{h}_0 \b{x}_{0}^*) + \b{e}.
\end{align}
\end{Lemma} 
This equivalent model to Proposition \ref{prop: newblindmodel} will the model worked with for the rest of the chapter.\\
\indent We aim to recover $(\b{h}_0, \b{x}_0)$ by solving the minimization problem\begin{align*} \label{eqn: blindminproblem}
\min_{(\b{h},\b{x})} F(\b{h},\b{x}), \quad F(\b{h},\b{x}) := \Vert\mathcal{A}(\b{h}\b{x}^*) - \b{y}\Vert^2 = \Vert\mathcal{A}(\b{h}\b{x}^* - \b{h}_0 \b{x}_{0}^*) - \b{e}\Vert^2.
\end{align*}
We also define
\begin{align*}
F_0(\b{h},\b{x}) := \Vert\mathcal{A}(\b{h}\b{x}^* - \b{h}_0 \b{x}_{0}^*)\Vert^2, \quad \delta = \delta(\b{h},\b{x}) = \dfrac{\Vert\b{h}\b{x}^* - \b{h}_{0}\b{x}_{0}^*\Vert_{F}}{d_0}.
\end{align*}
$F(\b{h},\b{x})$ is highly non-convex and thus attempts at minimization such as alternating minimization and gradient descent, can get easily trapped in some local minima.
\subsubsection{Main Theorems}
\begin{Theorem} \textbf{(Existence Of Unique Solution)} (\cite{ahmed2013blind}, Theorem 1) Fix $\alpha \geq 1$. Then there exists a constant $C_\alpha = O(\alpha)$, such that if
\begin{align*}
\max (K \cdot \mu_{\max}^{2}, N \cdot \mu_{h}^{2}) \leq \dfrac{d}{C_\alpha \log^3 d},
\end{align*}
then $\b{X}_0 = \b{h}_0 \b{x}_{0}^*$ is the unique solution to our minimization problem with probability $1 - O(d^{-\alpha + 1})$, thus we can separate $\b{y} = \b{f} * \b{g}$ up to a scalar multiple. When coherence is low, this is tight within a logarithmic factor, as we always have $\max (K,N) \leq d$.
\end{Theorem}
\begin{Theorem} \textbf{(Stability From Noise)} (\cite{ahmed2013blind}, Theorem 2) Let $\b{X}_0 = \b{h}_0 \b{m}_{0}^*$ and suppose the condition of previous theorem holds. We observe that 
\begin{align*}
\b{y}= \mathcal{A}(\b{X}_0) + \b{e},
\end{align*}
where $\b{e} \in \mathbb{R}^d$ is an unknown noise vector with $\Vert z\Vert_{2} \leq \delta$, and estimate $\b{X}_0$ by solving
\begin{align*}
\min \quad \Vert\b{X}\Vert_{*}, \qquad \text{subject to} \quad \Vert\widehat{y} - \mathcal{A}(\b{X})\Vert_{2} \leq \delta.
\end{align*}
Let $\lambda_{\min}, \lambda_{\max}$ be the smallest/largest non-zero eigenvalue of $\mathcal{A}\mathcal{A}^*$. Then with probability $1 - d^{-\alpha +1}$, the solution $\b{X}$ will obey
\begin{align*}
\Vert\b{X} - \b{X}_0\Vert_{F} \leq C \dfrac{\lambda_{\max}}{\lambda_{\min}} \sqrt{\min(K,N)} \delta,
\end{align*}
for a fixed constant $C$.
\end{Theorem}
\subsection{Wirtinger Gradient Descent} \label{sec: Wirtinger Gradient Descent}
In \cite{li2019rapid}, the approach is to solve the minimization problem (Equation \ref{eqn: blindminproblem}) using Wirtinger gradient descent. In this subsection, the algorithm is introduced as well as the main theorems which establish convergence of the proposed algorithm to the true solution.\\
\indent The algorithm consists of two parts: first an initial guess, and secondly, a variation of gradient descent, starting at the initial guess to converge to the true solution. Theoretical results are established for avoiding getting stuck in local minima.\\
\indent This is ensured by determining that the iterates are inside some properly chosen basin of attraction of the true solution.
\subsection{Basin of Attraction}
\begin{Proposition} \textbf{Basin of Attraction}: (Section 3.1, \cite{li2019rapid}) Three neighbourhoods are introduced whose intersection will form the basis of attraction of the solution:\\
(i) \textbf{Non-uniqueness}: Due to the scale ambiguity, for numerical stability we introduce the following neighbourhood
\begin{align*}
N_{L_0} := \{(\b{h},\b{x}) \mid \Vert\b{h}\Vert \leq 2\sqrt{L_0}, \Vert\b{x}\Vert \leq 2\sqrt{L_0}\}, \quad L_0 = \Vert\b{h}_0\Vert \cdot \Vert\b{x}_0\Vert.
\end{align*}
(ii) \textbf{Incoherence}: The number of measurements required for solving the blind deconvolution problem depends on how much $\b{h}_0$ is correlated with the rows of the
matrix $\b{B}$, with the hopes of minimizing the correlation. We define the incoherence between the rows of $\b{B}$ and $\b{h}_0$, via
\begin{align*}
\mu_{\b{h}}^{2} = \dfrac{d\Vert\b{B}\b{h}_0\Vert_{\infty}^2}{\Vert\b{h}_0\Vert^2}.
\end{align*}
To ensure that the  incoherence of the solution is under control, we introduce the neighborhood
\begin{align}
N_\mu := \{\b{h} \mid \sqrt{d} \Vert\b{B}\b{h}\Vert_{\infty} \leq 4 \sqrt{L_0}\mu\}, \quad \mu_{\b{h}} \leq \mu.
\end{align}
(iii) \textbf{Initial guess}: A carefully chosen initial guess is required due to the non-convexity of the function we wish to minimize. The distance to the true solution is defined via the following neighborhood
\begin{align}
N_\epsilon := \{(\b{h},\b{x}) \mid \Vert\b{hx}^* - \b{h}_0\b{x}_{0}^*\Vert_F \leq \epsilon L_0\}, \quad 0 < \epsilon \leq \dfrac{1}{15}.
\end{align}
Thus the basin of attraction is chosen as $N_{d_0} \cap N_\mu \cap N_\epsilon$, where the true solution lies.
\end{Proposition}
\begin{figure}[H]
\centering
\includegraphics[width=0.6\textwidth]{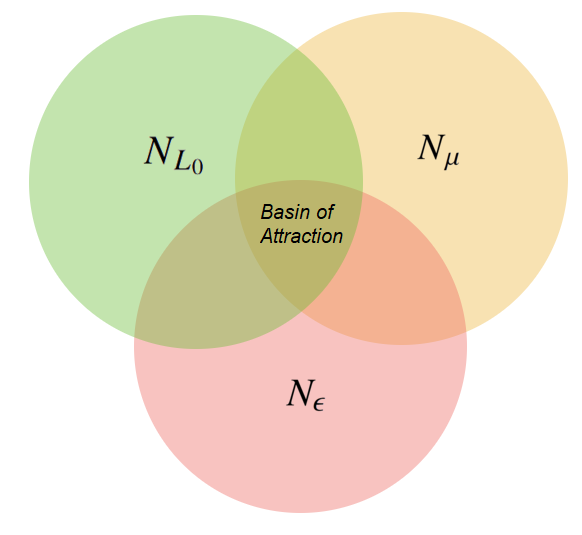}
\caption{Basin of Attraction: $N_{L_0} \cap N_\mu \cap N_\epsilon$}
\end{figure}
Our approach consists of two parts: We first construct an initial guess that is inside the basin of attraction $N_{L_0} \cap N_\mu \cap N_\epsilon$. We then apply a regularized Wirtinger gradient descent algorithm that will ensure that all the iterates remain inside  $N_{L_0} \cap N_\mu \cap N_\epsilon$. To achieve that, we add a regularizing function $G(\b{h},\b{x})$ to the objective function $F(\b{h},\b{x})$ to enforce that the iterates remain inside  $N_{L_0} \cap N_\mu$. Hence we aim the minimize the following regularized objective function, in order to solve the blind deconvolution problem:
\begin{align*}
\tilde{F}(\b{h}, \b{x}) := F(\b{h}, \b{x}) + G(\b{h}, \b{x}),
\end{align*}
where $F(\b{h},\b{x}) := \Vert\mathcal{A}(\b{h}\b{x}^* - \b{h}_0 \b{x}_{0}^*) - \b{e}\Vert^2$ is defined as before and $G(\b{h}, \b{x})$ is the penalty function, of the form
\begin{align*}
G(\b{h},\b{x}) := \rho  \bigg[ G_0 \bigg( \dfrac{\Vert\b{h}\Vert^2}{2L}\bigg) + G_0 \bigg( \dfrac{\Vert\b{x}\Vert^2}{2L}\bigg) + \sum_{\ell = 1}^{d} G_0 \bigg(\dfrac{d |\b{b}_{\ell}^* \b{h}|^2}{8 L \mu^2}\bigg)\bigg],
\end{align*}
where 
\begin{align*}
G_0 (z) := \max\{z - 1, 0 \}^2, \quad \rho \geq L^2 + 2\Vert\b{e}\Vert^2.
\end{align*}
It is assumed $\dfrac{9}{10}L_0 \leq L \leq \dfrac{11}{10}L_0$ and $\mu \geq \mu_h$.
\begin{Remark} The matrix $\mathcal{A}^* (\b{e}) = \sum_{k=1}^{d} \b{e}_k \b{b}_k \b{a}_{k}^*$ as a sum of d rank-1 random matrices, has nice concentration of measure properties. Asymptotically, $\Vert\mathcal{A}^*(e)\Vert$ converges to $0$ with rate $O(d^{1/2})$. Note that
\begin{align*}
F(\b{h},\b{x}) = \Vert\b{e}\Vert^2 + \Vert\mathcal{A}(\b{hx}^* - \b{h}_0 \b{x}_{0}^*)\Vert_{F}^2 - 2Re(\langle \mathcal{A}^* (\b{e}), \b{hx}^* - \b{h}_0 \b{x}_{0}^*\rangle).
\end{align*}
If one lets $d \longrightarrow \infty$, then $\Vert\b{e}\Vert^2 \sim \dfrac{\sigma^2 d_{0}^2}{2d}\rchi_{2d}^{2}$ will converge almost surely to $\sigma^2 d_{0}^2$\footnote{This is due to the law of large numbers} and the cross term $Re(\langle \b{hx}^* - \b{h}_{0}\b{x}_{0}^*, \mathcal{A}^*(\b{e}) \rangle)$ will converge to $0$. In other words, asymptotically
\begin{align*}
\lim_{d \to \infty} F(\b{h},\b{x}) = F_0 (\b{h},\b{x}) + \sigma^2 L_{0}^2,
\end{align*}
for all fixed $(\b{h},\b{x})$. This implies that if the number of measurements is large, then $F(\b{h},\b{x})$ behaves "almost like" $F_0 (\b{h},\b{x}) = \Vert\mathcal{A}(\b{hx}^* - \b{h}_{0}\b{x}_{0}^*)\Vert^2$, the noiseless version of $F(\b{h},\b{x})$. So for large $d$, we effectively ignore the noise.
\end{Remark}

\begin{Theorem} For any given $Z \in \mathbb{C}^{K \times N}$, we have that
\begin{align*}
\mathbb{E}(\mathcal{A}^*(\mathcal{A} (\b{Z}))) = \b{Z}.
\end{align*}
\end{Theorem} 
\begin{proof}
By linearity, sufficient to prove for $Z \in \mathbb{C}^{K \times N}$ where $Z_{i,j} = 1$, $0$ otherwise. Then we have that
\begin{align*}
\mathbb{E}((\mathcal{A}^*\mathcal{A} (\b{Z}))) = \mathbb{E} (\sum_{k=1}^{L}b_{ki}^* a_{kj} \b{b}_k \b{a}_{k}^*)) =  \sum_{k=1}^{d}(b_{ki}^* \b{b}_{k} \mathbb{E}(a_{kj}\b{a}_{k}^*)) = \sum_{k=1}^{d}(b_{ki}^* \b{b}_{k}) z_{j}^* = \b{e}_{i}\b{e}_{j}^* = Z,
\end{align*}
Thus we have that
\begin{align*}
\mathbb{E}( \mathcal{A}^* (\b{y}))= \mathbb{E} ( \mathcal{A}^*(\mathcal{A}(\b{h}_0 \b{x}_{0}^*) + \b{e}) = \mathbb{E}( \mathcal{A}^*(\mathcal{A} (\b{h}_0 \b{x}_{0}^*) )) + \mathbb{E}(\mathcal{A}^* (\b{e})) = \b{h}_0 \b{x}_{0}^*,
\end{align*}
since $\mathbb{E}(\mathcal{A}^* (\b{e}))$ by the definition of $\b{e}$.
\end{proof}
Hence it makes logical sense that the leading singular value and vectors of $\mathcal{A}^*(\b{y})$ would be a good approximation of $L_0$ and $(\b{h}_0,\b{x}_0)$ respectively.
\subsection{Algorithms} \label{sec: Blind Deconvolution Algorithms}
We can now state the algorithm for generating an initial estimate.
\begin{algorithm}[H]
\caption{Blind Deconvolution Initial Estimate} \label{alg: blindinitial}
\begin{algorithmic}
\Require Blind Deconvolutional measurements $\b{y}$, $K = \text{supp}(f)$. 
\Ensure Estimate underlying signal and blurring function.
\State 1) Compute $\mathcal{A}^* (\b{y})$ and find the leading singular value, left and right singular vectors of $\mathcal{A}^* (\b{y})$, denoted by $d$, $\tilde{\b{h}}_0$, and $\tilde{\b{x}}_0$ respectively.
\State 2) Solve the following optimization problem
\begin{align*}
\b{u_0} := \underset{\b{z}}{argmin}\Vert\b{z} - \sqrt{d}\tilde{\b{h}}_0\Vert^2, \; \;  \text{subject to} \; \sqrt{d} \Vert B\b{z}\Vert_{\infty} \leq 2\sqrt{L}\mu,
\end{align*}
and $\b{x_0} = \sqrt{L}\tilde{\b{x}}_0$.
\end{algorithmic}
\end{algorithm}
Since we are dealing with complex variables, for the gradient descent, Wirtinger derivatives are utilized. Since $\tilde{F}$ is a real-valued function, we only need to consider the derivative of $\tilde{F}$, with respect to $\bar{\b{h}}$ and $\bar{\b{x}}$, and the corresponding updates of $\b{h}$ and $\b{x}$ since
\begin{align*}
\dfrac{\partial \tilde{F}}{\partial \bar{\b{h}}} = \overline{\dfrac{\partial \tilde{F}}{\partial \b{h}}}, \quad \dfrac{\partial \tilde{F}}{\partial \bar{\b{x}}} = \overline{\dfrac{\partial \tilde{F}}{\partial \b{x}}}.
\end{align*}
In particular, we denote 
\begin{align}
\nabla \tilde{F}_{\b{h}} := \dfrac{\partial \tilde{F}}{\partial \bar{\b{h}}}, \quad \nabla \tilde{F}_{}\b{x} := \dfrac{\partial \tilde{F}}{\partial \bar{\b{x}}}.
\end{align}
We can now state the full algorithm.
\begin{algorithm}[H]
\caption{Wirtinger Gradient Descent Blind Deconvolution Algorithm} \label{alg: FullBlindDeconvolution}
\begin{algorithmic}
\Require Blind Deconvolutional measurements $\b{y}$, $K = \text{supp}(f)$.
\Ensure Estimate underlying signal and blurring function.
\State 1) Compute $\mathcal{A}^* (\b{y})$ and find the leading singular value, left and right singular vectors of $\mathcal{A}^* (\b{y})$, denoted by $d$, $\tilde{\b{h}}_0$, and $\tilde{\b{x}}_0$ respectively.
\State 2) Solve the following optimization problem
\begin{align*}
\b{u_0} := \underset{\b{z}}{argmin}\Vert\b{z} - \sqrt{L}\tilde{\b{h}}_0\Vert^2, \; \;  \text{subject to} \; \sqrt{d} \Vert B\b{z}\Vert_{\infty} \leq 2\sqrt{L}\mu,
\end{align*}
and $\b{x_0} = \sqrt{L}\tilde{\b{x}}_0$.
\State 3) Compute Wirtinger Gradient Descent
\While {halting criterion false}
\State $\b{u_t }= \b{u_{t-1}} - \upeta \nabla \tilde{F}_{\b{h}} (\b{u_{t-1}}, \b{u_{t-1}})$
\State $\b{v_t} = \b{v_{t-1}} - \upeta \nabla \tilde{F}_{\b{x}} (\b{v_{t-1}}, \b{v_{t-1}})$
\EndWhile
\State 4) Set $(\b{h},\b{x}) = (\b{u_t},\b{v_t})$
\end{algorithmic}
\end{algorithm}

In \cite{li2019rapid}, the authors show that with a carefully chosen initial guess $(\b{u}_0, \b{v}_0)$, running Wirtinger gradient descent to minimize  $\tilde{\b{F}}(h, x)$ will guarantee linear convergence of the sequence $(\b{u}_t, \b{v}_t)$ to the global minimum $(\b{h}_0, \b{x}_0)$ in the noiseless case, and also provide robust recovery in the presence of noise. The results are summarized in the following two theorems.
\subsection{Main Theorems} \label{sec: Blind Deconvolution Main Theorems}
\begin{Theorem} \textbf{(Main Theorem 1)} (\cite{li2019rapid}, Theorem 3.1)  The initialization obtained via Algorithm \ref{alg: blindinitial} satisfies 
\begin{align*}
(\b{u}_0, \b{v}_0) \in \dfrac{1}{\sqrt{3}} N_{L_0} \cap \dfrac{1}{\sqrt{3}}N_\mu \cap N_{\tfrac{2}{5}\epsilon}, \quad \dfrac{9}{10}L_0 \leq L \leq \dfrac{11}{10}L_0,
\end{align*}
with probability at least $1 - d^{-\gamma}$ if the number of measurements is sufficient large, that is
\begin{align*}
d \geq C_\gamma (\mu_{h}^2 + \sigma^2) \max\{K,N\} \log^2 d/\epsilon^2,
\end{align*}
where $\epsilon$ is a predetermined constant on $\bigg(0,\frac{1}{15}\bigg]$, and $C_\gamma$ is a constant only linearly depending on $\gamma$ with $\gamma \geq 1$.
\end{Theorem}
The following theorem establishes that as long as the initial guess lies inside the basin of attraction of the true solution, regularized gradient descent will converge to this solution (or to a nearby solution in case of noisy data).
\begin{Theorem} \textbf{(Main Theorem 2)}  (\cite{li2019rapid}, Theorem 3.2)  Assume that the initialization $(\b{u}_0, \b{v}_0) \in \frac{1}{\sqrt{3}}N_{L_0} \cap \frac{1}{\sqrt{3}}N_{\mu} \cap N_{\frac{2}{5}\epsilon}$, and that $d \geq C_\gamma (\mu^2 + \sigma^2) \max \{K,N\} \log^2 (L)/\epsilon^2$. Then Algorithm \ref{alg: FullBlindDeconvolution} will create a sequence $(\b{u}_t, \b{v}_t) \in N_{d_0} \cap N_{\mu} \cap N_{\epsilon}$ which converges geometrically to $(\b{h}_0, \b{x}_0)$ in the sense that with probability at least $1 - 4d^{-\gamma} - \dfrac{1}{\gamma} e^{-(K+N)}$, we have that
\begin{align*}
\max \{ sin \angle (\b{u}_t, \b{h}_0), sin \angle (\b{v}_t, \b{x}_0)\} \leq \dfrac{1}{L_t} \bigg( \dfrac{2}{3} (1 - \eta \omega)^{t/2} \eta L_0 + 50 \Vert\c{A}^* (\b{e})\Vert\bigg),
\end{align*}
and
\begin{align*}
|L_t - L_0| \leq \dfrac{2}{3} (1 - \eta \omega)^{t/2} \epsilon L_0 + 50\Vert\c{A}^* (\b{e})\Vert,
\end{align*}
where $L_t := \Vert\b{u}_t\Vert \cdot \Vert\b{v}_t\Vert, \omega > 0, \eta$ is the fixed stepsize. Here
\begin{align*}
\Vert\c{A}^* (\b{e})\Vert \leq C_0 \sigma d_0 \max \bigg\{ \sqrt{\dfrac{(\gamma + 1)\max\{K,N\} \log \; d}{d}}, \dfrac{(\gamma + d) \sqrt{KN} \log^2 \; d}{d}\bigg\},
\end{align*}
holds with probability $1 - d^{-\gamma}$.
\end{Theorem}
It has been shown with high probability that as long as the initial guess lies inside the basin of attraction of the true solution, Wirtinger gradient descent will converge towards the solution.
\subsection{Key Conditions} \label{sec: Blind Deconvolution Key Conditions}
\begin{Theorem} \label{thm: Four Key Conditions} \textbf{Four Key Conditions:}\\
(i) \textbf{(Local RIP Condition)}  (\cite{li2019rapid}, Condition 5.1) The following local Restricted Isometry Property (RIP) for A holds uniformly for all $(\b{h},\b{x})$ in the basin of attraction $(N_{L_0} \cap N_\mu \cap N_{\epsilon})$
\begin{align*}
\dfrac{3}{4} \Vert\b{hx}^* - \b{h}_0 \b{x}_{0}^*\Vert_{F}^2 \leq \Vert\mathcal{A}(\b{hx}^* - \b{h}_0\b{x}_{0}^*)\Vert^2 \leq \dfrac{5}{4} \Vert\b{hx}^* - \b{h}_0 \b{x}_{0}^*\Vert_{F}^2.
\end{align*}
(ii) \textbf{(Robustness Condition)}  (\cite{li2019rapid}, Condition 5.2) For the complex Gaussian noise $\b{e}$, with high probability
\begin{align*}
\Vert\mathcal{A}^* (\b{e})\Vert \leq \dfrac{\epsilon L_0}{10\sqrt{2}},
\end{align*}
for $d$ sufficiently large, that is, $d \geq C_\gamma (\frac{\sigma^2}{\epsilon^2} + \frac{\sigma}{\epsilon}) \max \{K,N\} log \; d$;\\
(iii) \textbf{(Local Regularity Condition)}  (\cite{li2019rapid}, Condition 5.3) There exists a regularity constant $\omega = \dfrac{d_0}{5000} > 0$ such that
\begin{align*}
\Vert\nabla \tilde{F}(\b{h},\b{x})\Vert^2 \geq \omega [\tilde{F}(\b{h},\b{x}) - c]_{+}, \quad c = \Vert\b{e}\Vert^2 + 1700\Vert\mathcal{A}^* (\b{e})\Vert^2,
\end{align*}
for all $(\b{h}, \b{x}) \in N_{L_0} \cap N_{\mu} \cap N_{\epsilon}$;\\
(iv) \textbf{(Local Smoothness Condition)}  (\cite{li2019rapid}, Condition 5.4) Denote $\b{z} := (\b{h},\b{x})$. There exists a constant $C_d$ such that
\begin{align*}
\Vert \nabla f(\b{z} + t\Delta \b{z}) - \nabla f(\b{z})\Vert \leq C_d t\Vert\Delta \b{z}\Vert, \quad 0 \leq t \leq 1,
\end{align*}
for all $\{(\b{z},\Delta \b{z}) \mid \b{z} + t\Delta \b{z} \in N_\epsilon \cap N_{\tilde{F}}, \forall 0 \leq t \leq 1\}$, i.e., the whole segment connecting $(\b{h},\b{x})$ and $\nabla(\b{h},\b{x})$ belongs to the non-convex set $N_\epsilon \cap N_{\tilde{F}}$.
\end{Theorem}

\section{Blind Ptychography} \label{sec: Blind Ptychography}
\subsection{Introduction}
A more recent area of study is blind ptychography, in which both the object and the mask are considered unknown, up to reasonable assumptions. The first successful recovery was given in (\cite{thibault2008high},\cite{thibault2009probe})
, further study into the sufficient overlap (\cite{bunk2008influence}, \cite{maiden2009improved}, \cite{maiden2017further}), and summarized in \cite{fannjiang2020blind}.\\
\indent Let $\mathbf{x}, \mathbf{m} \in \mathbb{C}^{d}$ denote the unknown sample and mask, respectively. We suppose that we have $d^2$ noisy ptychographic measurements of the form
\begin{align}
(\mathbf{Y})_{\ell,k} = |(\mathbf{F}(\mathbf{x} \circ S_k \mathbf{m}))_\ell|^2 + (\mathbf{N})_{\ell,k}, \quad (\ell,k) \in [d]_0 \times [d]_0,
\end{align}
where $S_k, \circ, \mathbf{F}$ denote $k^{\text{th}}$ circular shift, Hadamard product, and $d$-dimensional discrete Fourier transform, and $\mathbf{N}$ is the matrix of additive noise. By Theorem \ref{thm: BlindPtych Measurements}, we have shown we can rewrite the measurements as
\begin{align}
\bigg( \mathbf{Y}^T \mathbf{F}^{T}\bigg)_k = d \cdot (\mathbf{x} \circ S_k \bar{\mathbf{x}}) *  (\tilde{\mathbf{m}} \circ S_{-k} \bar{\tilde{\mathbf{m}}})  + \bigg( \mathbf{N}^T \mathbf{F}^{T}\bigg)_k,
\end{align}
where $*$ denotes the $d$-dimensional discrete convolution, and $\tilde{\mathbf{m}}$ denotes the reversal of $\mathbf{m}$ about its first entry. This is now a scaled blind deconvolution problem which has been studied in \cite{ahmed2013blind},\cite{li2019rapid}.
\subsection{Main Results}
\subsubsection{Recovering the Sample}
To recover the sample, we will need to assume that $\mathbf{x}$ belongs to a known subspace. Initially we solve algorithmically for the zero shift case ($k=0$) and then generalize the method to solve for the estimate which utilizes all the obtained shifts.\\
\indent Our assumptions are as follows: $\mathbf{x} \in \mathbb{C}^d$ unknown, $\mathbf{x} = C\mathbf{x}', C \in \mathbb{C}^{d \times N}, N \ll L$ known, $\mathbf{x}' \in \mathbb{C}^N$ or $\mathbb{R}^N$ unknown $\mathbf{m} \in \mathbb{C}^d$ unknown, $supp(\mathbf{m}) \subseteq [\delta]_0$, $K$ known, $\Vert\mathbf{m}\Vert_2$ known. Known noisy measurements $\mathbf{Y}$.\\
\indent Our first goal is to compute an estimate $\mathbf{x}_{est}$ of $\mathbf{x}$, true up to a global phase. We will use this estimate to then produce an estimate $\mathbf{m}_{est}$ of $\mathbf{m}$, again true up to a global phase. \\
\indent Firstly, we let $\mathbf{y}$ be the first column of $\frac{1}{\sqrt{d}} \cdot \mathbf{F}(\overline{\widetilde{(\mathbf{F}\mathbf{Y})^T}})$, $\mathbf{f} = \tilde{\mathbf{m}} \circ  \overline{\tilde{\mathbf{m}}}$ (so $\Vert\mathbf{f}\Vert_2$ known). We next set $\mathbf{g} = \mathbf{x} \circ \bar{\mathbf{x}}$ but to fully utilize the blind deconvolution algorithm, we will need a lemma concerning hadamard products of products of matrices. Firstly, we need define some products between matrices.
\begin{Definition} Let $\mathbf{A} = (A_{i,j}) \in \mathbb{C}^{m \times n}$ and $\mathbf{B} = (B_{i,j}) \in \mathbb{C}^{p \times q}$. Then the \textbf{Kronecker product} $\mathbf{A} \otimes \mathbf{B} \in \C^{mp \times nq}$ is defined by
\begin{align*}
(\mathbf{A} \otimes \mathbf{B})_{n(i-1) + k,q(j-1) + \ell} = A_{i,j}B_{k,\ell}.
\end{align*}
\end{Definition}
\begin{Definition} Let $\mathbf{A} \in \mathbb{C}^{m \times n}$ and $\mathbf{B} \in \mathbb{C}^{p \times n}$ with columns $\b{a}_i, \mathbf{b}_i$ for $i \in [n]_0$. Then the \textbf{Khatri-Rao product} $\mathbf{A} \bullet \mathbf{B} \in \mathbb{C}^{mp \times n}$ is defined by
\begin{align}
\mathbf{A} \bullet \mathbf{B} = [\mathbf{a}_0 \otimes \mathbf{b}_0 \; \mathbf{a}_1 \otimes \mathbf{b}_1 \ldots \mathbf{a}_{n-1} \otimes \mathbf{b}_{n-1}].
\end{align}
\end{Definition}
\begin{Definition} Let $\mathbf{A} \in \mathbb{C}^{m \times n}$ and $\mathbf{B} \in \mathbb{C}^{m \times p}$ be matrices with rows $\mathbf{a}_i, \mathbf{b}_i$ for $i \in [m]_0$. Then the \textbf{transposed Khatri-Rao product} (or \textbf{face-splitting product}), denoted $\odot$, is the matrix whose rows are Kronecker products of the columns of $\mathbf{A}$ and $\mathbf{B}$ i.e. the rows of $\mathbf{A} \odot \mathbf{B} \in \mathbb{C}^{m \times np}$ are given by
\begin{align*}
(\mathbf{A} \odot \mathbf{B})_i = \mathbf{a}_i \otimes \mathbf{b}_i, \quad i \in [m]_0.
\end{align*}
\end{Definition}
We then utilize the following lemma concerning the transposed Khatri-Rao product.
\begin{Lemma}[Theorem 1, \cite{slyusar1999family}] \label{lem: transkrproductlemma}
Let $\mathbf{A} \in \mathbb{C}^{m \times n}, \mathbf{B} \in \mathbb{C}^{n \times p}, \mathbf{C} \in \mathbb{C}^{m \times q},\mathbf{D} \in \mathbb{C}^{q \times p}$. Then we have that
\begin{align*}
(\mathbf{A}\mathbf{B}) \circ (\mathbf{C}\mathbf{D}) = (\mathbf{A} \bullet \mathbf{C})(\mathbf{B} \odot \mathbf{D}),
\end{align*}
where $\circ$ is the Hadamard product, $\bullet$ is the standard Khatri-Rao product, and $\odot$ is the transposed Khatri-Rao product.
\end{Lemma}
%\begin{proof} 
%\textcolor{red}{INSERT PROOF}
%\end{proof}
Thus by Lemma \ref{lem: transkrproductlemma} we have that for $\mathbf{g} = \mathbf{x} \circ \bar{\mathbf{x}} = \mathbf{C}\mathbf{x}' \circ \bar{\mathbf{C}}\bar{\mathbf{x}'}$. Then $g = \mathbf{C}'\mathbf{x}''$ where $\mathbf{C}' \in \mathbb{C}^{L \times N^2}, \mathbf{x}'' \in \mathbb{C}^{N^2}$ are given by
\begin{align*}
\mathbf{C}' = \mathbf{C} \bullet \bar{\mathbf{C}}, \quad \mathbf{x}'' = \mathbf{x}' \odot \bar{\mathbf{x}}'.
\end{align*}
We now compute RRR Blind Deconvolution (Algorithm \ref{alg: FullBlindDeconvolution}) with $\mathbf{y}, \mathbf{f}, \mathbf{g}, \mathbf{C}, K = \delta$ as above ($\mathbf{B}$ last K columns of DFT matrix) to obtain estimate for $\mathbf{x}' \odot \bar{\mathbf{x}}'$. Use angular synchronisation to solve for $\mathbf{x}'$, and thus solve for $\mathbf{x}$.
\begin{algorithm}[H]
\caption{Blind Ptychography (Zero Shift)} \label{Alg: BP ZS}
\begin{algorithmic}
\Require \\
1) $\mathbf{x} \in \mathbb{C}^d$ unknown, $\mathbf{x} = C\mathbf{x}', C \in \mathbb{C}^{d \times N}, N \ll L$ known, $\mathbf{x}' \in \mathbb{C}^N$ or $\mathbb{R}^N$ unknown.\\
2) $\mathbf{m} \in \mathbb{C}^d$ unknown, $supp(\mathbf{m}) \subseteq [\delta]_0$, $K$ known, $\Vert\mathbf{m}\Vert_2$ known
Known noisy measurements $\mathbf{Y}$.\\
3) Known noisy measurements $\mathbf{Y}$.
\Ensure Estimate $\mathbf{x}_{est}$ of $\mathbf{x}$ true up to a global phase.
\State 1) Let $\mathbf{y}$ be the first column of $\frac{1}{\sqrt{d}} \cdot \mathbf{F}(\overline{\widetilde{(\mathbf{F}\mathbf{Y})^T}})$, $\mathbf{f} = \tilde{\mathbf{m}} \circ  \overline{\tilde{\mathbf{m}}}$ (so $\Vert\mathbf{f}\Vert_2$ known)
\State 2) Let $\mathbf{g} = \mathbf{x} \circ \bar{\mathbf{x}} = \mathbf{C}\mathbf{x}' \circ \bar{\mathbf{C}}\bar{\mathbf{x}'}$. Then $g = \mathbf{C}'\mathbf{x}''$ where $\mathbf{C}' \in \mathbb{C}^{d \times N^2}, \mathbf{x}'' \in \mathbb{C}^{N^2}$ are given by
\vspace{-2mm}
\begin{align*}
\mathbf{C}' = \mathbf{C} \bullet \bar{\mathbf{C}}, \quad \mathbf{x}'' = \mathbf{x}' \odot \bar{\mathbf{x}}'.
\end{align*}
\vspace{-4mm}
\State 3) Compute RRR Blind Deconvolution (Algorithm 1 \& 2, \cite{li2019rapid}) with $\mathbf{y}, \mathbf{f}, \mathbf{g}, \mathbf{C}, K = \delta$ as above ($\mathbf{B}$ last K columns of DFT matrix) to obtain estimate for $\mathbf{x}' \odot \bar{\mathbf{x}}'$.
\State 4) Use angular synchronisation to solve for $\mathbf{x}'$, and thus compute $\mathbf{x}_{est}$.
\end{algorithmic}
\end{algorithm}

\subsubsection{Recovering the Mask}
Once the estimate of $\mathbf{x}$ has been found, denoted $\mathbf{x}_{est}$, we use this estimate to find $\mathbf{m}_{est}$. We first compute $\mathbf{g}_{est} = \mathbf{x}_{est} \circ \overline{\mathbf{x}}_{est}$, and then we use point-wise division to find
\begin{align}
\mathbf{F}(\tilde{\mathbf{m}} \circ \bar{\tilde{\mathbf{m}}}) = \dfrac{\overline{\mathbf{F}^{-1}(\widetilde{(\mathbf{F}\mathbf{Y})^T})}}{\mathbf{F}(\mathbf{x}_{est} \circ \overline{\mathbf{x}}_{est})}.
\end{align}
Then use an inverse Fourier transform, a reversal and then angular synchronization, similar to obtaining $\mathbf{x}_{est}$.

\begin{algorithm}[H]
\caption{Recovering The Mask} \label{Alg: Recovering The Mask}
\begin{algorithmic}
\Require
1)  $\mathbf{x}_{est}$ generated by Algorithm \ref{Alg: BP ZS}.\\
2) Known noisy measurements $\mathbf{Y}$.\\
3) $supp(\mathbf{m}) \subseteq [\delta]_0$, $K$ known, $\Vert\mathbf{m}\Vert_2$ known.
\Ensure Estimate $\mathbf{m}_{est}$ of $\mathbf{m}$ true up to a global phase.
\State 1) Compute $\mathbf{g}_{est} = \mathbf{x}_{est} \circ \overline{\mathbf{x}}_{est}$ and $2\delta-1$ perform point-wise divisions to obtain 
\begin{align}
\mathbf{F}(\tilde{\mathbf{m}}_{est} \circ S_{-k}\bar{\tilde{\mathbf{m}}}_{est}) = \dfrac{\overline{\mathbf{F}^{-1}(\widetilde{(\mathbf{F}\mathbf{Y})^T})_k}}{\mathbf{F}(\mathbf{x}_{est} \circ S_k\overline{\mathbf{x}}_{est})}.
\end{align}
\State 2) Compute inverse Fourier transform to obtain $\tilde{\mathbf{m}}_{est} \circ S_{-k}\bar{\tilde{\mathbf{m}}}_{est}$ and use these to form the diagonals of a banded matrix.
\State 3) Use angular synchronisation to solve for $\tilde{\mathbf{m}}_{est}$, and thus perform a reversal to compute $\mathbf{m}_{est}$.
\State 4) Let $\alpha = \dfrac{\Vert \mathbf{m}_{est} \Vert{}_2}{\Vert \mathbf{m} \Vert{}_2}$. Finally, let $\mathbf{x}_{est} = \alpha\mathbf{x}_{est}, \mathbf{m}_{est} = \alpha^{-1}\mathbf{m}_{est}$
\end{algorithmic}
\end{algorithm}

\subsubsection{Multiple Shifts}
To generalize the setup, we let $\mathbf{y}_{(k)}$ denote the $k^{\text{th}}$ column of $\frac{1}{\sqrt{d}} \cdot \mathbf{F}(\overline{\widetilde{(\mathbf{F}\mathbf{Y})^T}})$, $\mathbf{f}_{(k)} = \tilde{\mathbf{m}} \circ  S_{-k}\overline{\tilde{\mathbf{m}}}$. Let $\mathbf{g}_{(k)} = \mathbf{x} \circ S_k\bar{\mathbf{x}} = \mathbf{C}\mathbf{x}' \circ S_k\bar{\mathbf{C}}\bar{\mathbf{x}'}$. Then again by another application of Lemma \ref{lem: transkrproductlemma}, $\mathbf{g}_{(k)} = \mathbf{C}_{(k)}'\mathbf{x}''$ where $\mathbf{C}' \in \mathbb{C}^{d \times N^2}, \mathbf{x}'' \in \mathbb{C}^{N^2}$ are given by
\begin{align*}
\mathbf{C}_{(k)}' = \mathbf{C} \; \bullet \; S_k\bar{\mathbf{C}}, \quad 0 \leq k \leq K, d-K+1 \leq k \leq d, \quad \mathbf{x}'' = \mathbf{x}' \odot \bar{\mathbf{x}}' = vec(\mathbf{x}'(\mathbf{x}')^*).
\end{align*}
We then perform $2\delta-1$ blind deconvolutions to obtain $2\delta - 1$ estimates of $\mathbf{x}$ and $\mathbf{m}$, labelled $\mathbf{x}_{est}^{i}, \mathbf{m}_{est}^{j}$ respectively for $i,j \in [2\delta-1]_0$.\\
\indent Ideally, we would want to select the estimates which generates the minimum error for each $\mathbf{x}$ and $\mathbf{m}$ but that implies prior knowledge of $\mathbf{x}$ and $\mathbf{m}$. Instead, compute $(2\delta-1)^2$ estimates of the Fourier measurements by
\begin{align}
(\mathbf{Y}_{est}^{i,j})_{\ell,k} = |(\mathbf{F}(\mathbf{x}_{est}^{i} \circ S_k \mathbf{m}_{est}^{j}))_{\ell}|^2, \quad i,j \in [2\delta-1]_0.
\end{align}
We then compute the associated error
\begin{align}
(i',j') = \underset{(i,j)}{\text{argmin}} \dfrac{\Vert\mathbf{Y}_{est}^{i,j} - \mathbf{Y}\Vert_{F}^2}{\Vert\mathbf{Y}\Vert_{F}^2}, \quad i,j \in [2\delta-1]_0.
\end{align}
Then let $\mathbf{x}_{est} = \mathbf{x}_{est}^{i'}, \mathbf{m}_{est} = \mathbf{m}_{est}^{j'}$. 

\begin{algorithm}[H]
\caption{Blind Ptychography (Multiple Shifts)} \label{Alg: BP MS}
\begin{algorithmic}
\Require \\
1) $\mathbf{x} \in \mathbb{C}^d$ unknown, $\mathbf{x} = C\mathbf{x}', C \in \mathbb{C}^{d \times N}, N \ll d$ known, $\mathbf{x}' \in \mathbb{C}^N$ or $\mathbb{R}^N$ unknown.\\
2) $\mathbf{m} \in \mathbb{C}^d$ unknown, $supp(\mathbf{m}) \subseteq [\delta]_0$, $K$ known, $\Vert\mathbf{m}\Vert_2$.\\
3) Known noisy measurements $\mathbf{Y}$.
\Ensure Estimate $\mathbf{x}_{est}$ of $\mathbf{x}$, true up to a global phase
\State 1) Let $\mathbf{y}_{(k)}$ denote the $k^{\text{th}}$ column of $\frac{1}{\sqrt{d}} \cdot \mathbf{F}(\overline{\widetilde{(\mathbf{F}\mathbf{Y})^T}})$, $\mathbf{f}_{(k)} = \tilde{\mathbf{m}} \circ  S_{-k}\overline{\tilde{\mathbf{m}}}$ (so $\Vert\mathbf{f}_{(k)}\Vert_2$ known).
\State 2) Let $\mathbf{g}_{(k)} = \mathbf{x} \circ S_k\bar{\mathbf{x}} = \mathbf{C}\mathbf{x}' \circ S_k\bar{\mathbf{C}}\bar{\mathbf{x}'}$. Then $g_{(k)} = \mathbf{C}'\mathbf{x}''$ where $\mathbf{C}' \in \mathbb{C}^{d \times N^2}, \mathbf{x}'' \in \mathbb{C}^{N^2}$ are given by
\begin{align*}
\mathbf{C}_{(k)}' = C \; \bullet \; S_k\bar{C}, \quad 0 \leq k \leq K, d-K+1 \leq K \leq d, \quad \mathbf{x}'' = \mathbf{x}' \odot \bar{\mathbf{x}}'.
\end{align*}
%where $\odot$ is the transposed Khatri-Rao product, $\bullet$ Khatri–Rao product
\State 3) Perform $2\delta-1$ RRR Blind Deconvolutions (Algorithm 1 \& 2, \cite{li2019rapid}) with $\mathbf{y}_{(k)}, \mathbf{f}_{(k)}, \mathbf{g}_{(k)}, \mathbf{C}$, as above to obtain $2\delta-1$ estimates for $\mathbf{x}' \odot \bar{\mathbf{x}}'$.
\State 4) Use angular synchronisation to solve for $2\delta-1$ estimates $\mathbf{x}'_{est}$, and thus solve for $2\delta-1$ estimates $\mathbf{x}_{est}^i = \mathbf{C}\mathbf{x}'_{est}, i \in [2\delta-1]_0$.
\State 5) Use these estimates $\mathbf{x}_{est}^{i}$ to compute $2\delta -1$ estimates $\mathbf{m}_{est}^{j}, j \in [2\delta-1]_0$.
\State 6) Let $\alpha^i = \dfrac{\Vert \mathbf{m}_{est}^i \Vert{}_2}{\Vert \mathbf{m} \Vert{}_2}$, and for $i \in [2\delta-1]_0$, let $\mathbf{x}_{est}^i = \alpha^i\mathbf{x}_{est}^i, \mathbf{m}_{est}^i = \dfrac{1}{\alpha^i}\mathbf{m}_{est}^i$
\State 7) Compute $(2\delta-1)^2$ estimates of the Fourier measurements by
\begin{align}
(\mathbf{Y}_{est}^{i,j})_{\ell,k} = |(\mathbf{F}(\mathbf{x}_{est}^{i} \circ S_k \mathbf{m}_{est}^{j}))_{\ell}|^2, \quad i,j \in [2\delta-1]_0.
\end{align}
We then compute the associated error $(i',j') = \underset{(i,j)}{\text{argmin}} \dfrac{\Vert\mathbf{Y}_{est}^{i,j} - \mathbf{Y}\Vert_{F}^2}{\Vert\mathbf{Y}\Vert_{F}^2}, \quad i,j \in [2\delta-1]_0$.
\State 8) Let $\mathbf{x}_{est} = \mathbf{x}_{est}^{i'}, \mathbf{m}_{est} = \mathbf{m}_{est}^{j'}$.
\end{algorithmic}
\end{algorithm}

\subsection{Numerical Simulations}
All simulations were performed using MATLAB R2021b on an Intel desktop with a 2.60GHz i7-10750H CPU and 16GB DDR4 2933MHz memory. All code used to generate the figures below is publicly available at https://github.com/MarkPhilipRoach/BlindPtychography.\\
\indent To be more precise, we have defined the immeasurable (in practice since $\mathbf{x}$ and $\mathbf{m}$ are both unknown) estimates
\begin{align*}
\text{Max Shift}^{(x)} &= \underset{\mathbf{x}_{est}^{i}}{\text{argmax}} \Vert\mathbf{x} - \mathbf{x}_{est}^{i}\Vert_{2}^{2},&\\
\text{Max Shift}^{(m)} &= \underset{\mathbf{m}_{est}^{j}}{\text{argmax}} \Vert\mathbf{m} - \mathbf{m}_{est}^{j}\Vert_{2}^{2}, \quad i,j \in [2\delta-1]_0,&\\
\text{Min Shift}^{(x)} &= \underset{\mathbf{x}_{est}^{i}}{\text{argmin}} \Vert\mathbf{x} - \mathbf{x}_{est}^{i}\Vert_{2}^{2},&\\ 
\text{Min Shift}^{(m)} &= \underset{\mathbf{m}_{est}^{j}}{\text{argmin}} \Vert\mathbf{m} - \mathbf{m}_{est}^{j}\Vert_{2}^{2}, \quad i,j \in [2\delta-1]_0.&
\end{align*}
and the measurable estimates. First, $\text{No Shift}^{(x)}$ and $\text{No Shift}^{(m)}$ refer to the zero shift estimates outlined in Algorithm \ref{Alg: BP ZS}. Secondly, the estimates achieved in Algorithm \ref{Alg: BP MS}
\begin{align*}
&(\text{Argmin Shift}^{(x)}, \text{Argmin Shift}^{(m)}) = (\mathbf{x}^{i'}, \mathbf{m}^{j'}),&\\
&(i',j') =  \underset{(i,j)}{\text{argmin}} \dfrac{\Vert\mathbf{Y}_{est}^{i,j} - \mathbf{Y}\Vert_{F}^2}{\Vert\mathbf{Y}\Vert_{F}^2}, \quad i,j \in [2\delta-1]_0.&
\end{align*}

\begin{figure}[H]
\centering
\includegraphics[width=1\textwidth]{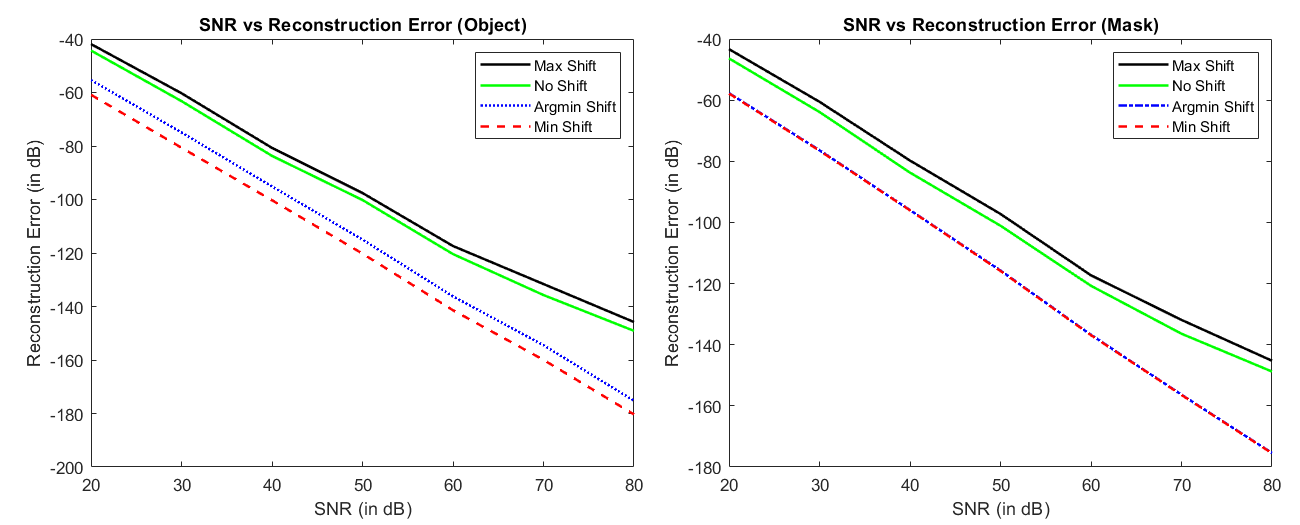}
\caption{$d = 2^6, K = \delta = \log_2 d, N = 4, \mathbf{C}$ complex Gaussian. Max shift refers to the maximum error achieved from a blind deconvolution of a particular shift. Min shift refers to the maximum error achieved from a blind deconvolution of a particular shift. Argmin Shift refers to the choice of object and mask chosen in Step 6 of Algorithm \ref{Alg: BP MS}. Averaged over 100 simulations. 1000 iterations. }
\label{fig: SNRvsErrorN=4}
\end{figure}
Figure \ref{fig: SNRvsErrorN=4} demonstrates robust recovery under noise. It also demonstrates the impact of performing the $2\delta-1$ blind deconvolutions and taking the Argmin Shift, versus simply taking the non-shifted object and mask. It also demonstrates how closely the reconstructions error from Argmin Shift and Min Shift are, in particular for the mask. Figure \ref{fig: SNRvsErrorN=6} demonstrates the impact even more, showing that with a higher value for the known subspace, the more accurate the Argmin Shift and Min Shift are, as well as demonstrating the large difference between the Max Shift and Min Shift.
\begin{figure}[H]
\centering
\includegraphics[width=1\textwidth]{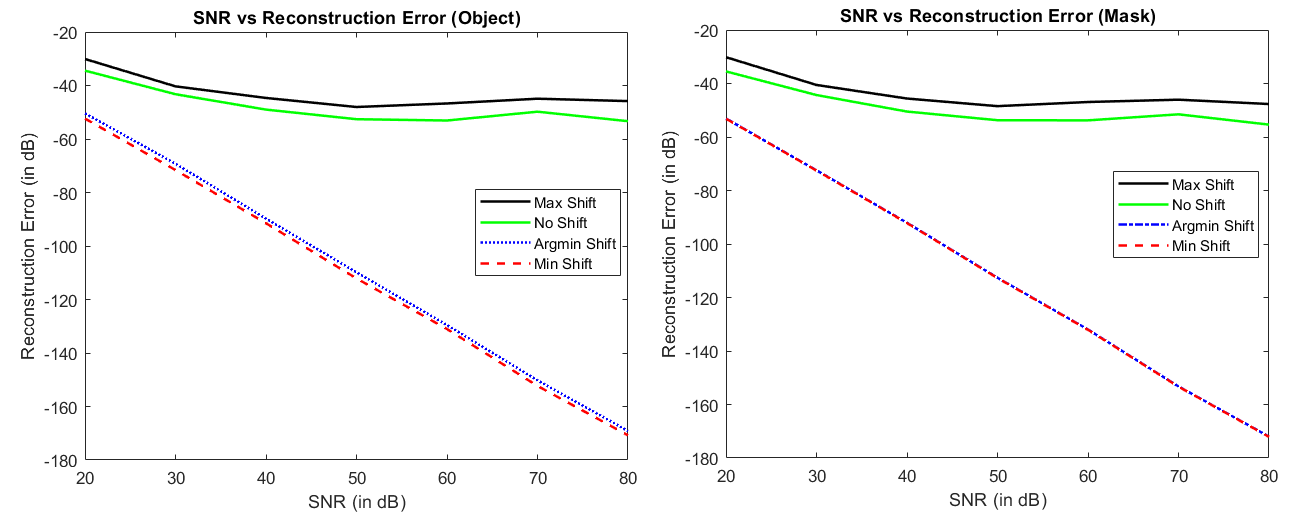}
\caption{$d = 2^6, K = \delta = \log_2 d, N = 6, \mathbf{C}$ complex Gaussian. Max shift refers to the maximum error achieved from a blind deconvolution of a particular shift. Min shift refers to the maximum error achieved from a blind deconvolution of a particular shift. Argmin Shift refers to the choice of object and mask chosen in Step 6 of Algorithm \ref{Alg: BP MS}. Averaged over 100 simulations. 1000 iterations. }
\label{fig: SNRvsErrorN=6}
\end{figure}
The following figures demonstrate recovery against additional noise, with varying $\delta$ and $N$.
\begin{figure}[H]
\centering
\includegraphics[width=1\textwidth]{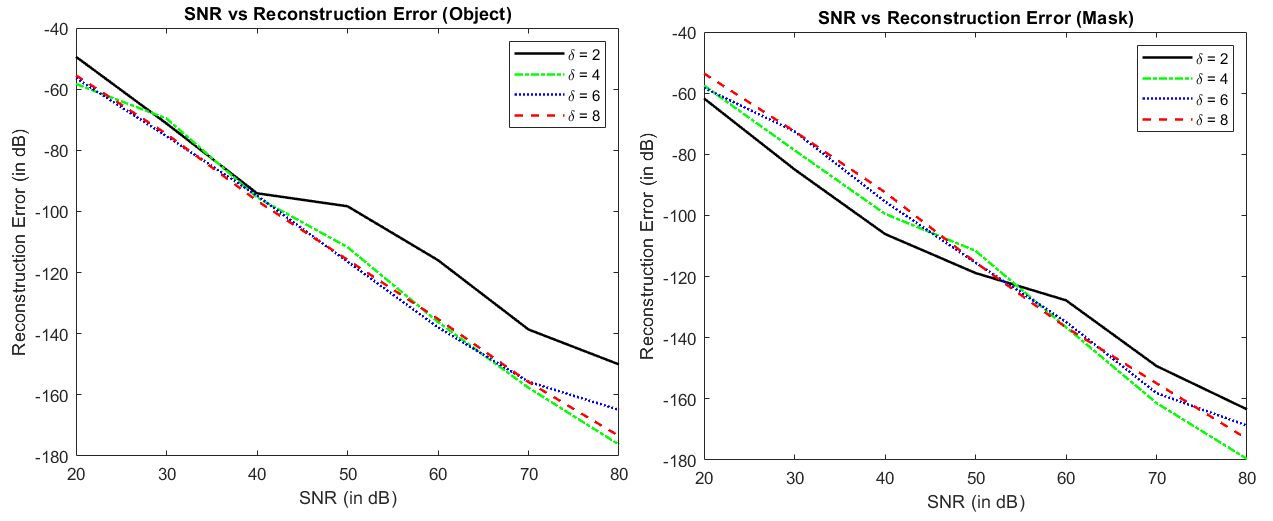}
\caption{$d = 2^6, N = 4, \mathbf{C}$ complex Gaussian. Application of Algorithm \ref{Alg: BP MS} with varying $K = \delta$.}
\label{fig: SNRvsErrorVaryingDelta}
\end{figure}
\begin{figure}[H]
\centering
\includegraphics[width=1\textwidth]{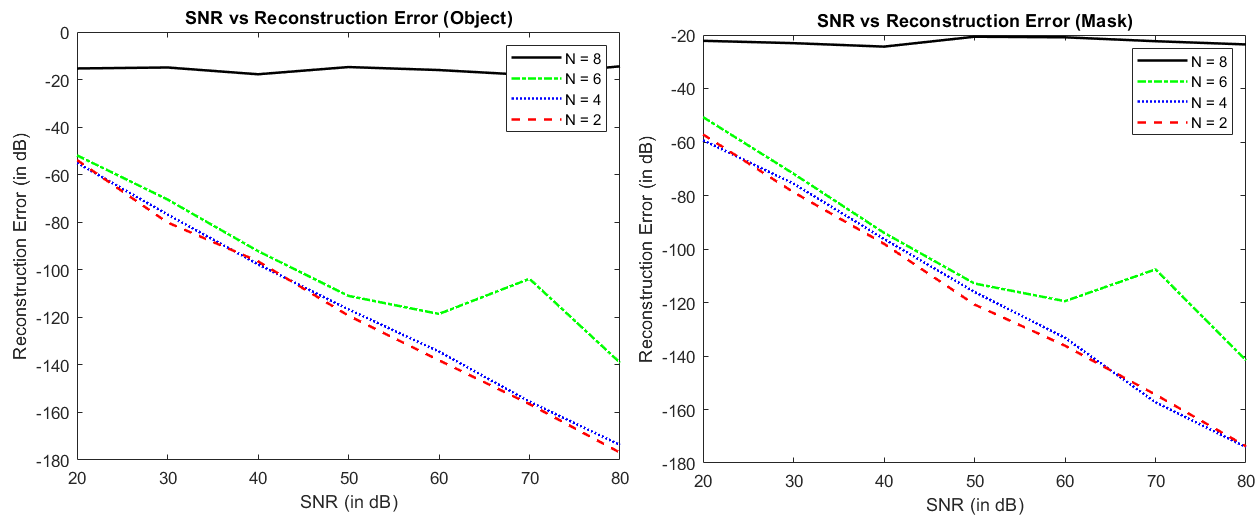}
\caption{$d = 2^6, K = \delta = 6, \mathbf{C}$ complex Gaussian. Application of Algorithm \ref{Alg: BP MS} with varying $N$.}
\label{fig: SNRvsErrorVaryingN}
\end{figure}
Next, we consider the frequency of the chosen index from performing the argmin function (step 6 of Algorithm \ref{Alg: BP MS}) compared to the true minimizing indices for the object and mask separately.\\
\indent Firstly, we have the frequency of the argmin indices.\begin{figure}[H]
\centering
\includegraphics[width=1\textwidth]{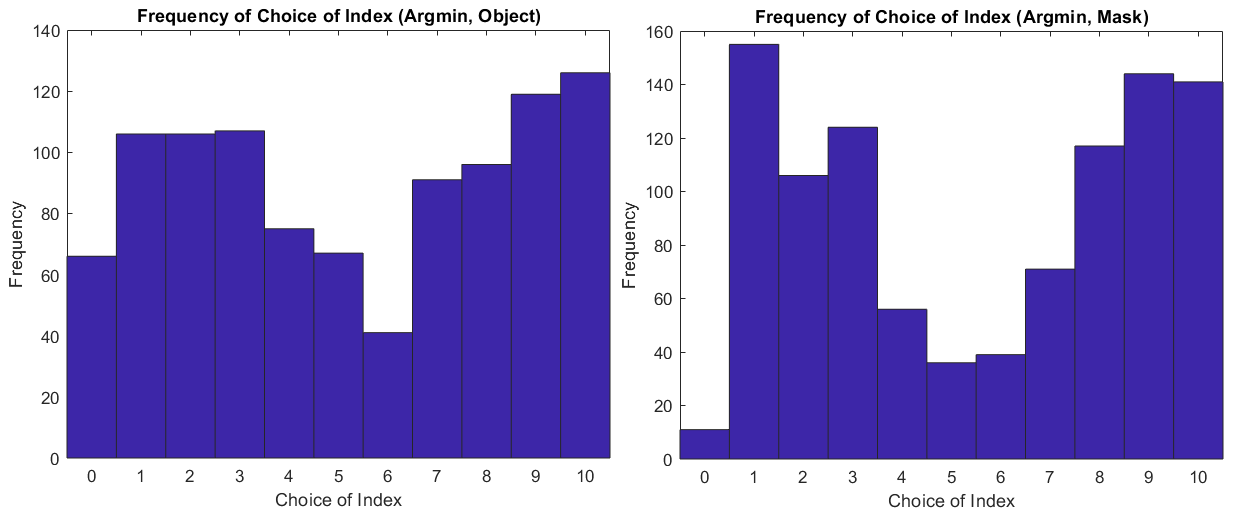}
\caption{$d = 2^6, \delta = 6, N = 4, \mathbf{C}$ complex Gaussian. 1000 simulations. Frequency of index being chosen to compute $\text{Argmin Shift}^{(x)}$ and $\text{Argmin Shift}^{(m)}$.}
\label{fig: FrequencyArgmin}
\end{figure}
Secondly, we have the frequency of the min shift for both of the object and mask. Both of Figure \ref{fig: FrequencyArgmin} and Figure \ref{fig: FrequencyMin} were computed on the same 1000 tests.
\begin{figure}[H]
\centering
\includegraphics[width=1\textwidth]{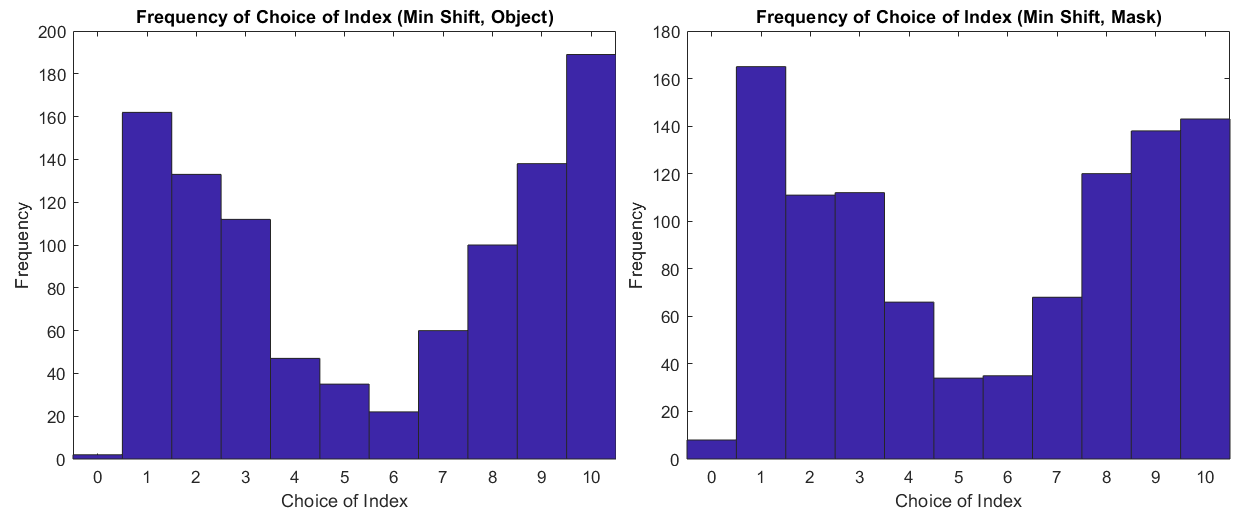}
\caption{$d = 2^6, \delta = 6, N = 4, \mathbf{C}$ complex Gaussian. 1000 simulations. Frequency of index being chosen to compute $\text{Min Shift}^{(x)}$ and $\text{Min Shift}^{(m)}$.}
\label{fig: FrequencyMin}
\end{figure}
Finally, we plot these choice of indices for both the Argmin Shift and Min Shift onto a two dimensional plot. 
\begin{figure}[H]
\centering
\includegraphics[width=1\textwidth]{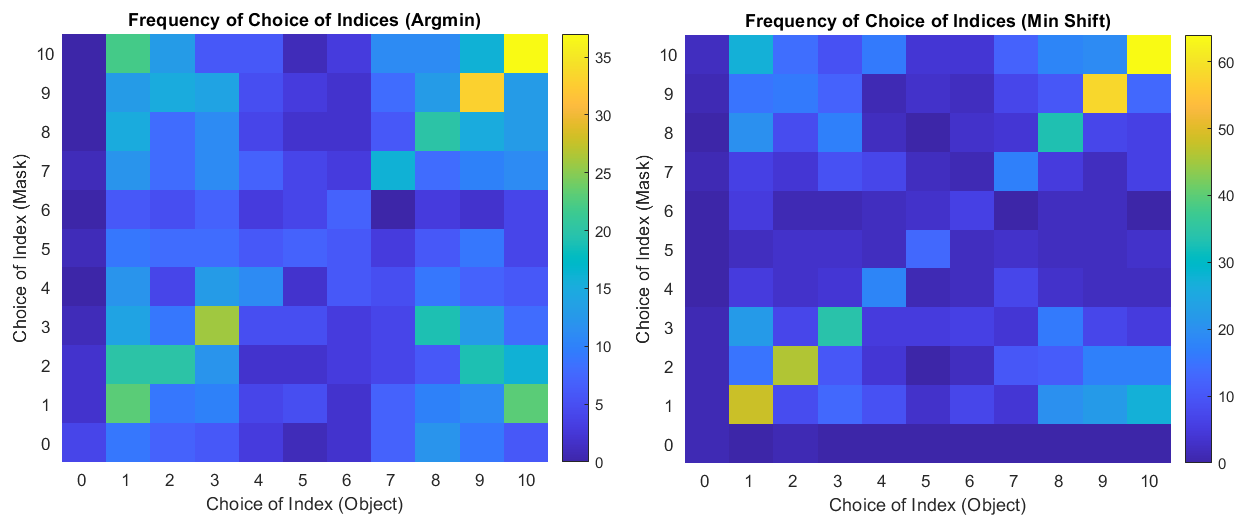}
\caption{$d = 2^6, \delta = 6, N = 4, \mathbf{C}$ complex Gaussian. 1000 simulations. Frequency of indices being chosen to compute $(\text{Argmin Shift}^{(x)}, \text{Argmin Shift}^{(m)})$ and $(\text{Min Shift}^{(x)}, \text{Min Shift}^{(m)})$.}
\label{fig: FrequencyIndices}
\end{figure}

\section{Conclusions and Future Work}
\label{sec: BP Conclusion}

We have introduced an algorithm for recovering a specimen of interest from blind far-field ptychographic measurements. This algorithm relies on reformulating the measurements so that they resemble widely-studied blind deconvolutional measurements. This leads to transposed Khatri-Rao product estimates of our specimen which are then able to be recovered by angular synchronization. We then use these estimates in applying inverse Fourier transforms, point-wise division, and angular synchronization to recover estimates for the mask. Finally, we use a best error estimate sorting algorithm to find the final estimate of both the specimen and mask. As shown in numerical results, Algorithm \ref{Alg: BP MS} recovers both the sample and mask within a good margin of error. It also provides stability under noise. A further goal for this research would be to adapt the existing recovery guarantee theorems for the selected blind deconvolutional recovery algorithm, in which the assumed Gaussian matrix $\mathbf{C}$ is replaced with Khatri-Rao matrix $\mathbf{C}_{(k)}' = \mathbf{C} \; \bullet \; S_k\bar{\mathbf{C}}$. In particular, this would mean providing alternate inequalities for the four key conditions laid out in Theorem \ref{thm: Four Key Conditions}.

%\begin{itemize}
%\item Assume the specimen and mask are unknown in the ptychography model and use blind deconvolution algorithms to solve the problem
%\item Changing the underlying assumptions in the blind deconvolution model to more better fit real world scenarios i.e. not assuming the Gaussian property of $C$
%\item Modeling other, more realistic, choices for C
%\item Investigating the relationship between assumptions of the blind ptychography model and the assumptions of the blind deconvolution model
%\item Explore possible other blind deconvolution algorithms
%\item Explore other real world scenarios in which a blind ptychographic model would be appropriate
%\item Compare uniqueness guarantees with already established papers by Fannjiang $\&$ others
%\end{itemize}

\begin{appendices}
\newpage
\section{Sub-Sampling}\label{BP Appendix}
In this section, we discuss sub-sampling lemmas that can be used in conjunction with Algorithm
\ref{alg: WDD Ptych}.
%\begin{figure}[H]
%\centering
%\includegraphics[width=0.6\textwidth]{SingleAnglePtychography}
%\caption{\cite{hruszkewycz2017high},\cite{argonne} A visualization of single-angle Bragg ptychography being used to examine the affect of stress on atoms (Image by Robert Horn/Argonne National Laboratory)}
%\end{figure}
In many cases, an illumination of the sample can cause damage to the sample, and applying the illumination beam (which can be highly irradiative) repeatedly at a single point can destroy it. Considering the risks to the sample and the costs of operating the measurement equipment, there are strong incentives to reduce the number of illuminations applied to any object.
\begin{Definition} Let $s \in \mathbb{N}$ such that $s \mid d$. We define the \textbf{sub-sampling operator} $Z_s : \mathbb{C}^{d} \longrightarrow \mathbb{C}^{\frac{d}{s}}$ defined component-wise via
\begin{align}
(Z_s \b{x})_n := x_{n \cdot s}, \quad \forall n \in [d/s]_0
\end{align}
\end{Definition}
We now have an aliasing lemma which allows us to see the impact of performing the Fourier transform on a sub-sampled specimen.
\begin{Lemma} \textbf{(Aliasing)} (\cite{merhi2019phase}, Lemma 2.0.1.) Let $s \in \mathbb{N}$ such that $s \mid d$, $\b{x} \in \mathbb{C}^{d}, \omega \in [\frac{d}{s}]_0$. Then we have that
\begin{align}
\bigg(F_{\frac{d}{s}} (Z_s \b{x})\bigg)_\omega = \dfrac{1}{s} \sum_{r = 0}^{s - 1} \hat{\b{x}}_{\omega - r \frac{d}{s}} 
\end{align}
\end{Lemma}
\begin{proof} Let $d \in \mathbb{N}$ and suppose $s \in \mathbb{N}$ divides $d$. Let $\b{x} \in \mathbb{C}^{d}$ and $\omega \in [\tfrac{d}{s}]_0$ be arbitrary. By the definition of the discrete Fourier transform and sub-sampling operator, we have that
\begin{align}
\bigg(F_{\frac{d}{s}} (Z_s \b{x})\bigg)_\omega = \sum_{n=0}^{\frac{d}{s} - 1} (Z_s \b{x})_n e^{-\frac{2\pi i n \omega}{d/s}} =  \sum_{n=0}^{\frac{d}{s} - 1} x_{ns} e^{-\frac{2\pi i n \omega s}{d}}
\end{align}
By the inverse DFT and by collecting terms, we have that
\begin{align}
\sum_{n=0}^{\frac{d}{s} - 1} x_{ns} e^{-\frac{2\pi i n \omega s}{d}} = \dfrac{1}{d} \sum_{n=0}^{\frac{d}{s} - 1} \bigg( \sum_{r=0}^{d - 1} \hat{x}_r e^{\frac{2\pi i rn s}{d}}\bigg)e^{-\frac{2\pi i n (r - \omega) s}{d}}
\end{align}
By treating this as a sum of DFTs, we then have that
\begin{align}
\dfrac{1}{d} \sum_{n=0}^{\frac{d}{s} - 1} \bigg( \sum_{r=0}^{d - 1} \hat{x}_r e^{\frac{2\pi i rn s}{d}}\bigg)e^{-\frac{2\pi i n (r - \omega) s}{d}} = \dfrac{1}{s} \sum_{r = 0}^{s - 1} \hat{x}_{\omega + r\frac{d}{s}} = \dfrac{1}{s} \sum_{r = 0}^{s - 1} \hat{x}_{\omega - r\frac{d}{s}}
\end{align}
\end{proof}
Before we start looking at aliased WDD, we need to introduce a lemma which will show the effect of taking a Fourier transform of an autocorrelation.
\begin{Lemma} \textbf{(Fourier Transform Of Autocorrelation)} (\cite{merhi2019phase}, Lemma 2.0.2.)  Let $\b{x} \in \mathbb{C}^{d}$ and $\alpha, \omega \in [d]_0$. Then
\begin{align}
\bigg(\b{F}_d (\b{x} \circ S_\omega \bar{\b{x}})\bigg)_\alpha = \dfrac{1}{d} e^{2\pi i \omega \alpha/d} \bigg(\b{F}_d (\hat{\b{x}} \circ S_{-\alpha} \bar{\hat{\b{x}}})\bigg)_\omega
\end{align}
\end{Lemma}
\begin{proof} Let $\b{x} \in \mathbb{C}^{d}$ and let $\alpha, \omega \in [d]_0$ be arbitrary. By the convolution theorem, we have that
\begin{align}
\bigg(\b{F}_d (\b{x} \circ S_\omega \bar{\b{x}})\bigg)_\alpha = \dfrac{1}{d} (\hat{\b{x}} *_d \b{F}_d (S_\omega \bar{\b{x}}))_\alpha
\end{align}
By technical equality (iii), we can revert the Fourier transform of the shift operator to the modulation operator of the Fourier transform
\begin{align}
\dfrac{1}{d} (\hat{\b{x}} *_d \b{F}_d (S_\omega \bar{\b{x}}))_\alpha &= \dfrac{1}{d} (\hat{\b{x}} *_d (W_\omega \hat{\bar{\b{x}}})_\alpha &= \dfrac{1}{d} \sum_{n=0}^{d - 1} \hat{x}_n (W_\omega \hat{\bar{\b{x}}})_{\alpha - n} = \dfrac{1}{d} \sum_{n=0}^{d - 1} \hat{x}_n \hat{\bar{x}}_{\alpha - n} e^{\frac{2\pi i \omega(\alpha - n)}{d}}
\end{align}
with the latter equalities being the definition of the convolution and modulation. By applying reversals and using that $\tilde{\tilde{\b{x}}} = \b{x}$, we have that
\begin{align}
\dfrac{1}{d} \sum_{n=0}^{d - 1} \hat{x}_n \hat{\bar{x}}_{\alpha - n} e^{\frac{2\pi i \omega(\alpha - n)}{d}} = \dfrac{1}{d} e^{\frac{2\pi i \omega\alpha}{d}}\sum_{n=0}^{d - 1} \hat{x}_n \tilde{\hat{\bar{x}}}_{n - \alpha} e^{-\frac{2\pi i \omega\alpha}{d}}
\end{align}
Finally by applying technical equality (vi) and using the definition of the shift operator and Hadamard product, we have that
\begin{align}
\dfrac{1}{d} e^{\frac{2\pi i \omega\alpha}{d}}\sum_{n=0}^{d - 1} \hat{x}_n \tilde{\hat{\bar{x}}}_{n - \alpha} e^{-\frac{2\pi i \omega\alpha}{d}} = \dfrac{1}{d} e^{\frac{2\pi i \omega\alpha}{d}}\sum_{n=0}^{d - 1} \hat{x}_n \bar{\hat{x}}_{n - \alpha} e^{-\frac{2\pi i \omega\alpha}{d}} = \dfrac{1}{d} e^{2\pi i \omega \alpha/d} \bigg(\b{F}_d (\hat{\b{x}} \circ S_{-\alpha} \bar{\hat{\b{x}}})\bigg)_\omega
\end{align}
\end{proof}

\subsection{Sub-Sampling In Frequency}
We will first look at sub-sampling in frequency.
\begin{Definition} Let $K$ be a positive factor of $d$, and assume that the data is measured at $K$ equally spaced Fourier modes. We denote the set of Fourier modes of step-size $\frac{d}{k}$ by
\begin{align}
\mathcal{K} = \frac{d}{K}[K]_0 = \bigg\{0, \frac{d}{K},\frac{2d}{K}, \ldots, d - \frac{d}{K}\bigg\}
\end{align}
\end{Definition}
\begin{Definition} Let $\b{A} \in \mathbb{C}^{d \times d}$ with columns $\b{a}_j$, $K \mid d$. We denote by $\b{A}_{K,d} \in \mathbb{C}^{K \times d}$ the sub-matrix of $\b{A}$ whose $\ell^{\text{th}}$ column is equal to $Z_{\frac{d}{K}} (\b{a}_\ell)$.
\end{Definition}
With these definitions, we will now convert the sub-sampled measurements into a more solvable form.
\begin{Lemma} (\cite{merhi2019phase}, Lemma 2.1.1.) Suppose that the noisy spectrogram measurements are collected on a subset $\mathcal{K} \subseteq [d]_0$ of $K$ equally space Fourier modes. Then for any $\omega \in [K]_0$
\begin{align*}
\bigg((\b{F}_K \b{Y}_{K,d})^T\bigg)_\omega = K\sum_{r = 0}^{\frac{d}{K} - 1}(\b{x} \circ S_{\omega - rK}\bar{\b{x}}) *_d (\tilde{\b{m}} \circ S_{rK - \omega} \bar{\tilde{\b{m}}}) + \bigg((\b{F}_K \b{N}_{K,d})^T\bigg)_\omega 
\end{align*}
where $\b{Y}_{K,d} \in \mathbb{C}^{K \times d} - \b{N}_{K,d} \in \mathbb{C}^{K \times d}$ is the matrix of sub-sampled noiseless $K \cdot d$ measurements.
\end{Lemma}
\begin{proof} For $\ell \in [d]_0$, the $\ell^{\text{th}}$ column of the matrix $\b{Y}$ is
\begin{align}
\b{y}_\ell = \b{F}_d \bigg( (\b{x} \circ S_{-\ell}\b{m}) *_d (\bar{\tilde{\b{x}}} \circ S_\ell \bar{\tilde{\b{m}}})\bigg) + \eta_\ell
\end{align}
and thus for any $\alpha \in [K]_0$
\begin{align}
\bigg(Z_{\frac{d}{K}} (\b{y}_\ell)\bigg)_\alpha = \bigg(\b{F}_d \bigg( (\b{x} \circ S_{-\ell}\b{m}) *_d (\bar{\tilde{\b{x}}} \circ S_\ell \bar{\tilde{\b{m}}})\bigg)\bigg)_{\alpha \frac{d}{K}} + \bigg(Z_{\frac{d}{K}} (\b{\eta}_\ell)\bigg)_\alpha
\end{align}
and by aliasing lemma (with $s = \frac{d}{K}$)
\begin{align}
\bigg(\b{F}_K\bigg(Z_{\frac{d}{K}} (\b{y}_\ell)\bigg)_\alpha\bigg)_\omega &= \dfrac{K}{d}\sum_{r = 0}^{\frac{d}{K} - 1} (\hat{\b{y}}_\ell)_{\omega - rK}&\\
&= d \cdot \dfrac{K}{d}\sum_{r = 0}^{\frac{d}{K} - 1} \bigg( (\b{x} \circ S_{\omega - rK}\bar{\b{x}}) *_d (\tilde{\b{m}} \circ S_{rK - \omega} \bar{\tilde{\b{m}}})\bigg)_\ell + \bigg(\b{F}_K\bigg(Z_{\frac{d}{K}} (\b{\eta}_\ell)\bigg)_\alpha\bigg)_\omega&
\end{align}
The $\ell^{\text{th}}$ column of $\b{Y}_{K,d} \in \mathbb{C}^{K \times d}$ is equal to $Z_{\frac{d}{K}} (\b{y}_\ell)$. Then for any $\omega \in [K]_0$, the $\omega^{\text{th}}$ column of $(\b{F}_K Y_{K,d})^T \in \mathbb{C}^{d \times K}$  may be computed as 
\begin{align}
\bigg( (\b{F}_K \b{Y}_{K,d})^T\bigg)_\omega = K\sum_{r = 0}^{\frac{d}{K} - 1}(\b{x} \circ S_{\omega - rK}\bar{\b{x}}) *_d (\tilde{\b{m}} \circ S_{rK - \omega} \bar{\tilde{\b{m}}}) + \bigg( (\b{F}_K \b{N}_{K,d})^T\bigg)_\omega \in \mathbb{C}^{d}
\end{align}
\end{proof}
\subsection{Sub-Sampling In Frequency And Space}
We will now look at sub-sampling in both frequency and space.
\begin{Definition} Let $L$ be a positive factor of $d$. Suppose measurements are collected at $L$ equally spaced physical shifts of step-size $\frac{d}{L}$. We denote the set of shifts by $\mathcal{L}$, that is
\begin{align}
\mathcal{L} = \frac{d}{L}[L]_0 = \bigg\{0, \frac{d}{L},\frac{2d}{L}, \ldots, d - \frac{d}{L}\bigg\}
\end{align}
\end{Definition}
\begin{Definition} Let $\b{A} \in \mathbb{C}^{d \times d}$, $L \mid d$. We denote by $\b{A}_{d,L} \in \mathbb{C}^{d \times L}$ the sub-matrix of $\b{A}$ whose rows are those of $\b{A}$, sub-sampled in step-size $\frac{d}{L}$.
\end{Definition}
We will now prove a similar lemma as before, but now we will sub-sample in both frequency and space.
\begin{Lemma} (\cite{merhi2019phase}, Lemma 2.1.2.)  Suppose we have noisy spectrogram measurements collected on a subset $\mathcal{K} \subseteq [d]_0$ of $K$ equally spaced frequencies and a subset $\mathcal{L} \subseteq [d]_0$ of $L$ equally spaced physical shifts. Then for any $\omega \in [K]_0, \alpha \in [L]_0$
\begin{align*}
\bigg(\b{F}_L \bigg(\b{Y}_{K,L}^T (\b{F}_{K}^T)_\omega\bigg)\bigg)_\alpha &= \dfrac{KL}{d} \sum_{r = 0}^{\frac{d}{K} - 1} \sum_{\ell = 0}^{\frac{d}{L} - 1}\bigg(\b{F}_d ((\b{x} \circ S_{\omega - rK}\bar{\b{x}}) *_d (\tilde{\b{m}} \circ S_{rk - \omega} \bar{\tilde{\b{m}}}))\bigg)_{\alpha - \ell L}&\\
&+ \bigg(\b{F}_L \bigg(\b{N}_{K,L}^T (\b{F}_{K}^T)_\omega\bigg)\bigg)_\alpha&
\end{align*}
where $\b{Y}_{K,L} - \b{N}_{K,L} \in \mathbb{C}^{K \times L}$ is the matrix of sub-sampled noiseless $K \cdot L$ measurements.
\end{Lemma}
\begin{proof} For fixed $\ell \in [d]_0, \omega \in [K]_0$, we have computed
\begin{align*}
\bigg(\b{F}_K (Z_{\frac{d}{K}} (\b{y}_\ell)) \bigg)_\omega = K\sum_{r = 0}^{\frac{d}{K} - 1} \bigg((\b{x} \circ S_{\omega - rK}\bar{\b{x}}) *_d \b{F}_d (\tilde{\b{m}} \circ S_{rk _ \omega} \bar{\tilde{\b{m}}})\bigg)_\ell
\end{align*}
Fix $\omega \in [K]_0$, and define the vector $\b{p}_\omega \in \mathbb{C}^{L}$ by
\begin{align*}
(\b{p}_\omega)_\ell := \bigg( \b{F}_K (Z_{\frac{d}{K}} (\b{y}_{\ell \frac{d}{L}}))\bigg)_\omega + \bigg( \b{F}_K (Z_{\frac{d}{K}} (\b{\eta}_{\ell \frac{d}{L}}))\bigg)_\omega, \quad \forall \ell \in [L]_0
\end{align*}
Note that the rows of $\b{Y}_{K,L}, \b{N}_{K,L} \in \mathbb{C}^{K \times L}$ are those of $\b{Y}_{K,L}, \b{N}_{K,L} \in \mathbb{C}^{K \times d}$, sub-sampled in step-size of $\frac{d}{L}$. Thus
\begin{align*}
(\b{p}_\omega)_\ell =\bigg( (\b{Y}_{K,L})^T (\b{F}_{K}^T)_\omega\bigg)_\ell + \bigg( (\b{Y}_{K,L})^T (\b{F}_{K}^T)_\omega\bigg)_\ell
\end{align*}
where $(\b{F}_{K}^{T})_\omega \in \mathbb{C}^{K}$ is the $\omega^{\text{th}}$ column of $\b{F}_{K}^{T}$. Therefore
\begin{align*}
\b{p}_\omega = \b{Y}_{K,L}^T (\b{F}_{K}^T)_\omega +  \b{N}_{K,L}^T (\b{F}_{K}^T)_\omega \in \mathbb{C}^{L}, \quad \forall \omega \in [K]_0
\end{align*}
For any $\ell \in [L]_0$, we have
\begin{align*}
(\b{p}_\omega)_\ell &= K\sum_{r = 0}^{\frac{d}{K} - 1} \bigg((\b{x} \circ S_{\omega - rK}\bar{\b{x}}) *_d \b{F}_d (\tilde{\b{m}} \circ S_{rk _ \omega} \bar{\tilde{\b{m}}})\bigg)_{\ell\frac{d}{L}} +  \b{N}_{K,L}^T (\b{F}_{K}^T)_\omega&\\
&= K \cdot \bigg(Z_{\frac{d}{L}} \bigg( \sum_{r = 0}^{\frac{d}{K} - 1} (\b{x} \circ S_{\omega - rK}\bar{\b{x}}) *_d \b{F}_d (\tilde{\b{m}} \circ S_{rk _ \omega} \bar{\tilde{\b{m}}})\bigg)\bigg)_\ell +  \b{N}_{K,L}^T (\b{F}_{K}^T)_\omega&
\end{align*}
and for any $\alpha \in [L]_0$, by aliasing, we have that
\begin{align*}
(\b{F}_L \b{p}_\omega)_\alpha = \dfrac{KL}{d} \sum_{r = 0}^{\frac{d}{K} - 1} \sum_{\ell = 0}^{\frac{d}{L} - 1}\bigg(\b{F}_d ((\b{x} \circ S_{\omega - rK}\bar{\b{x}}) *_d (\tilde{\b{m}} \circ S_{rk - \omega} \bar{\tilde{\b{m}}}))\bigg)_{\alpha - \ell L} + \bigg(\b{F}_L \bigg(\b{N}_{K,L}^T (\b{F}_{K}^T)_\omega\bigg)\bigg)_\alpha
\end{align*}
\end{proof}
\clearpage
\section{Alternative Approach}
\vspace{-5mm}
In this section we discuss the convex relaxation approach studied in \cite{ahmed2013blind}.
%\vspace{-2mm}
\subsection{Convex Relaxation} \label{sec: Convex Relaxation}
In \cite{ahmed2013blind}, the approach is to solve a convex version of the problem. Given $\b{y} \in \mathbb{C}^L$, their goal is to find $\b{h} \in \mathbb{R}^k$, $\b{x} \in \mathbb{R}^N$ that are consistent with the observations. Making no additional assumptions other than the dimensions, the way to choose between multiple feasible points is by solving using least-squares. That is,
\begin{align*}
\text{minimize}_{\b{u},\b{v}} \; \Vert\b{u}\Vert_{2}^{2} + \Vert\b{v}\Vert_{2}^{2} \qquad \text{subject to} \;\; \b{y}(\ell) = \langle \b{c}_\ell, \b{u}\rangle \langle \b{v}, \b{b}_\ell \rangle, \quad 1 \leq \ell \leq L
\end{align*}
This is a non-convex quadratic optimization problem. The cost function is convex, but the quadratic equality constraints mean that the feasible set is non-convex. The dual of this minimization problem is the SDP and taking the dual again will give us a convex program 
\begin{align*}
\min_{\b{W}_1, \b{W}_2, \b{X}} \tfrac{1}{2}tr(\b{W}_1) + \tfrac{1}{2} tr(\b{W}_2) \qquad \text{subject to} 
\begin{bmatrix}
\b{W}_1 & \b{X}\\
\b{X}^* & \b{W}_2\\
\end{bmatrix}
\succeq 0, \b{y}= \mathcal{A}(\b{X})
\end{align*}
which is equivalent to 
\begin{align*}
\min \Vert\b{X}\Vert_*  \qquad \text{subject to} \;\; \b{y} = \mathcal{A}(\b{X})
\end{align*}
where $\Vert\b{X}\Vert^* = tr(\sqrt{\b{X}^* \b{X}})$ denotes the nuclear norm.
In \cite{ahmed2013blind}, they achieved guarantees for relatively large $K$ and $N$, when $B$ is incoherent in the Fourier domain, and when $C$ is generic. We can now outline the algorithm from \cite{ahmed2013blind}.
\begin{algorithm}[H]
\caption{ Convex Relaxed Blind Deconvolution Algorithm}
\begin{algorithmic}
\Require Normalized Fourier measurement $\b{y}$, 
\Ensure Estimate underlying signal and blurring function
\State 1) Compute $\mathcal{A}^* (\b{y})$
\State 2) Find the leading singular value, left and right singular vectors of $\mathcal{A}^* (\b{y})$, denoted by $d$, $\tilde{\b{h_0}}$, and $\tilde{\b{x_0}}$ respectively
\State 3) Let $\b{X}_0 = \tilde{\b{h}_{0}}\tilde{\b{x}_{0}}^*$ denote the initial estimate and solve the following optimization problem
\begin{alignat}{3}
&\min \quad &&\Vert\b{X}\Vert_{*}&&\nonumber\\
&\text{subject to} \quad &&\Vert\b{y} - \c{A}(\b{X})\Vert \leq \delta&&
\end{alignat}
where $\Vert\cdot\Vert_{*}$ denotes the nuclear norm and $\Vert\b{e}\Vert_{2} \leq \delta$
\State Return $(\b{h},\b{x})$ for $\b{X} = \b{hx}^*$
\end{algorithmic}
\end{algorithm}

\end{appendices}

\bibliography{ThesisBib}
\end{document}